\documentclass[reqno]{amsart}
\usepackage{url}
\usepackage{amssymb,amsthm,amsfonts,amstext}
\usepackage{amsmath}
\usepackage{mathptmx}       
\usepackage{helvet}         
\usepackage{courier}        
\usepackage{graphicx}
\usepackage{enumerate}
\usepackage[active]{srcltx}
\usepackage{mathrsfs}
\usepackage{tikz}
\usepackage{graphicx}
\usepackage{epsfig}
\usetikzlibrary{shapes}
\usepackage{color}
\usepackage[ansinew]{inputenc} 
\newtheorem{theorem}{Theorem}[section]
\newtheorem{definition}{Definition}[section]
\newtheorem{lemma}[theorem]{Lemma}
\newtheorem{proposition}[theorem]{Proposition}
\newtheorem{corollary}[theorem]{Corollary}
\newtheorem{remark}[theorem]{Remark}

\newcommand{\mc}[1]{{\mathcal #1}}

\newcommand{\bb}[1]{{\mathbb #1}}

\newcommand{\eps}{\varepsilon}

\newcommand{\Y}{\mathcal{Y}}
\newcommand{\<}{\langle}
\renewcommand{\>}{\rangle}
\newcommand{\p}{\partial}
\newcommand{\pfrac}[2]{\genfrac{}{}{}{1}{#1}{#2}}

\newcommand{\at}[2]{\genfrac{}{}{0pt}{}{#1}{#2}}

\newcommand{\A}{\Delta_\beta}
\newcommand{\B}{\nabla_\beta}


\date{}

\begin{document}

\title[Equilibrium fluctuations]{Phase transition in equilibrium fluctuations\\ of symmetric slowed
exclusion}

\author{Tertuliano Franco}
\address{UFBA\\
 Instituto de Matem\'atica, Campus de Ondina, Av. Adhemar de Barros, S/N. CEP 40170-110\\
Salvador, Brasil}
\curraddr{}
\email{tertu@impa.br}
\thanks{}

\author{Patr\'{\i}cia Gon\c{c}alves}
\address{Departamento de Matem\'atica, PUC-RIO, Rua Marqu\^es de S\~ao Vicente, no. 225, 22453-900, Rio de Janeiro, Rj-Brazil and CMAT, Centro de Matem\'atica da Universidade do Minho, Campus de Gualtar, 4710-057, Braga, Portugal}
\curraddr{}
\email{patricia@mat.puc-rio.br}
\thanks{}

\author{Adriana Neumann}
\address{UFRGS, Instituto de Matem\'atica, Campus do Vale, Av. Bento Gon\c calves, 9500. CEP 91509-900, Porto Alegre, Brasil}
\curraddr{}
\email{aneumann@impa.br}
\thanks{}

\begin{abstract}
 We analyze the equilibrium fluctuations of density, current and tagged particle  in symmetric exclusion with a slow bond. The system evolves in the
one-dimensional lattice and the jump rate is everywhere equal to one except  at the slow bond where it is $\alpha n^{-\beta}$, with $\alpha>0$, $\beta\in[0,+\infty]$ and $n$ is the
scaling parameter.
Depending on the regime of $\beta$, we find three different behaviors for the limiting fluctuations whose covariances are
explicitly computed.
In particular, for the critical value $\beta=1$, starting a tagged particle near the slow bond,  we obtain a family of gaussian processes indexed in $\alpha$,
interpolating a fractional Brownian motion of Hurst exponent $1/4$ and the degenerate process equal to zero.
\end{abstract}

\maketitle

\section{Introduction}
The exclusion process is a standard  interacting particle system, widely studied in Probability and Statistical
Mechanics. Informally, such model corresponds to particles performing continuous time random walks in a lattice, except when a particle tries to
jump to an already occupied site. In such case, the jump is forbidden and the particle has to wait a new random time.

There is an intensive research on the behavior of exclusion processes in many different aspects and from varied points of view. In particular,  on the behavior
of exclusion processes in random/non homogeneous medium, see for instance \cite{fjl,fsv,fgn,fl}.

In this paper we analyze the fluctuations of the one-dimensional symmetric exclusion process with a slow bond, for which the
hydrodynamic limit was treated in \cite{fgn,fgn2}. The dynamics of this model can be described as follows. On the
one-dimensional lattice, it is allowed at most one particle per site. To each bond is associated a Poisson clock. When this clock
rings, the occupation variables at the vertices of the bond are interchanged with a certain rate.
Of course, if both the sites are occupied or empty, nothing
happens. All bonds have a Poisson clock of parameter one, except one special bond, the \emph{slow bond}, in which the Poisson clock has
parameter $\alpha n^{-\beta}$, where $\alpha>0$, $\beta\in[0,+\infty]$ and $n$ is the integer scaling parameter. At the end $n$ is lead to infinity. The
process starts from the equilibrium measure, namely a Bernoulli product measure of parameter $\rho\in {(0,1)}$, and it is seen in the diffusive time scale,
or else, in times of order $n^2$.

We are concerned with the  fluctuations,  that is, the Central Limit Theorem (C.L.T.) for the density, the current of particles through a fixed bond and the tagged particle. Such results are well known for the classical symmetric exclusion,
where all the Poisson clocks have parameter one. For the density, the fluctuations are given by a generalized Ornstein-Uhlenbeck process, while the fluctuations of the current and the tagged particle are both given by the fractional
Brownian motion of Hurst exponent $1/4$, see \cite{kl,ps}.

The introduction of the slow bond changes dramatically the scenario. Not so  intuitively, the value $\beta=1$ is critical.
For $\beta\in[0,1)$, we obtain {\em ipsis litteris} the same results for the fluctuations of the symmetric exclusion just mentioned. This means that, in this case, the jump rate at the slow bond is not sufficiently strong in order to change the macroscopic behavior of the system.  Nevertheless, the proof of this result is not straightforward and requires a Local Replacement which is sharp for this regime of $\beta$. For $\beta\in (1,+\infty]$, it is proved here that the fluctuations
of the density are driven by the semigroup of the heat equation with Neumann's boundary conditions. This means that for this regime of $\beta$,
the slow bond splits the system into two separate regions in which the macroscopic dynamics evolves independently.

Finally, at the critical value $\beta=1$, we prove that the density fluctuation field converges to a generalized Ornstein-Uhlenbeck
process driven by the semigroup of the partial differential equation
\begin{equation}\label{robin}
\left\{
\begin{array}{ll}
 \partial_t u(t,x) = \; \partial^2_{xx} u(t,x), &t \geq 0,\, x \in \mathbb R\backslash\{0\}\\
\p_x u(t,0^+)=\p_x u(t,0^-)=\alpha\{u(t,0^+)-u(t,0^-) \},  &t \geq 0\\
 u(0,x) = \; g(x), &x \in \mathbb R
\end{array}
\right.
\end{equation}
if the slow bond is located near the origin.
 If the slow bond is located elsewhere, the result is the same, but with the boundary conditions stated above for the corresponding macroscopic point. We remark that last equation is similar to the heat equation with a boundary condition of Robin's type, but relating the positive and negative half-lines.
 Notice that, for this regime of $\beta$, the parameter $\alpha$ survives in the limit. In the case $\alpha=1$, we characterize explicitly the Ornstein-Uhlenbeck process obtained in  \cite{fsv}. More precisely,  the authors of \cite{fsv} consider the process evolving on $\bb T_n$, take a general measure $W$ and prove that the density fluctuation field converges to an Ornstein-Uhlenbeck process, which is not explicit. Taking $W$ as the sum of  the Lebesgue measure and a delta of Dirac and considering the  process evolving in infinite volume, we give an explicit description of the aforementioned process.

Knowing the density fluctuations, we obtain, for the three regimes of $\beta$, the corresponding current
fluctuations and we compute explicitly the covariances for the limiting gaussian processes. It is of worth to remark the behavior of the fluctuations of the current through the slow bond. For $\beta\in{[0,1)}$ we get a fractional Brownian motion of Hurst exponent $1/4$ and for $\beta\in{(1,+\infty]}$ we get the degenerate process equal to zero. For $\beta=1$, the current fluctuations are given by a family of
gaussian processes indexed in $\alpha$ interpolating the fractional Brownian motion of Hurst exponent $1/4$ and the degenerate
process equal to zero.  By this, we mean that we can recover these two processes from the case $\beta=1$ by taking the limit as $\alpha\to +\infty$ or as $\alpha\to 0$, respectively, being the convergence in the sense of finite dimensional distributions.

Lastly, as a consequence of the previous result, it is straightforward to obtain the C.L.T. for a tagged particle. In this case, we consider as initial measure the Bernoulli product measure conditioned to have a particle at a given site. Therefore, the system is no longer in equilibrium, but anyhow we can use the previous result to deduce the behavior of a tagged particle in this non-equilibrium situation. Following \cite{g,jl} and since we are in  dimension one, the aforementioned result follows from relating the position of a tagged particle with the current and the density of particles. 

The paper is divided as follows. In Section \ref{s2}, we introduce notation and state the results. In Section \ref{s3}, we present the C.L.T. for the density of particles. In Section \ref{s4}, we get an
explicit formula for the semigroup of \eqref{robin}. In Section \ref{s5} we give a
martingale characterization of the generalized Ornstein-Uhlenbeck processes obtained in the fluctuations of the density of particles. In Section \ref{s6}, we prove the
C.L.T. for the current. Section \ref{s7} contains some useful estimates that we will use along the text.

\section{Definitions and main results}\label{s2}
\subsection{The model}

The symmetric simple exclusion process with conductances $\xi^n_{x,x+1}\geq{0}$ is a Markov process $\{\eta_t:t\geq{0}\}$, with configuration space $\Omega:=\{0,1\}^{\bb Z}$. We denote by $\eta$ the
configurations of the state space $\Omega$ so that $\eta(x)=0$, if the
site $x$ is vacant, and $\eta(x)=1$, if the site $x$ is occupied. Its infinitesimal generator $\mc L_{n}$
acts on local functions $f:\Omega\rightarrow \bb{R}$ as
\begin{equation}\label{ln}
(\mc L_{n}f)(\eta)=\sum_{x\in \bb Z}\,\xi^{n}_{x,x+1}\,\Big[f(\eta^{x,x+1})-f(\eta)\Big]\,,
\end{equation}
where $\eta^{x,x+1}$ is the configuration obtained from $\eta$ by exchanging the occupation variables $\eta(x)$ and $\eta(x+1)$:
\begin{equation*}
(\eta^{x,x+1})(y)=\left\{\begin{array}{cl}
\eta(x+1),& \mbox{if}\,\,\, y=x\,,\\
\eta(x),& \mbox{if} \,\,\,y=x+1\,,\\
\eta(y),& \mbox{otherwise.}
\end{array}
\right.
\end{equation*}
We define the symmetric exclusion with a slow bond at $\{-1,0\}$ by taking the conductances as
\begin{equation*}
\xi^{n}_{x,x+1}\;=\;\left\{\begin{array}{cl}
\alpha n^{-\beta}, &  \mbox{if}\,\,\,\,x=-1\,,\\
1, &\mbox{otherwise\,.}
\end{array}
\right.
\end{equation*}

We notice that when $\beta=0$ and $\alpha=1$, the process becomes the well known symmetric simple exclusion process. We are interested in
analyzing the
behavior of the process when $\alpha>0$ and $\beta\in(0,+\infty]$.

A simple computation shows that the Bernoulli product measures $\{\nu_\rho : 0\le \rho \le 1\}$ are invariant, in fact reversible, for the symmetric simple exclusion process with conductances, in particular also for the considered process. More precisely,
$\nu_\rho$ is
a product measure over $\Omega$ with marginals given by $\nu_\rho \{\eta : \eta(x) =1\} \;=\; \rho$, for $x$ in $\bb Z$.

Denote by $\{\eta_{tn^2} : t\ge 0\}$ the Markov process on $\Omega$ associated to the generator $n^2\mc L_n$.
Let $\mc D(\bb R_+, \Omega)$ be the path space of c\`adl\`ag trajectories (continuous from the right with limits from the left) with values in
$\Omega$. For a measure $\mu_n$ on $\Omega$, denote by $\bb P_{\mu_n}^\beta$ the probability measure on
$\mc D(\bb R_+, \Omega)$ induced by the initial state $\mu_n$ and the Markov process $\{\eta_{tn^2} : t\ge 0\}$. Expectation with respect to $\bb
P_{\mu_n}^\beta$ will be denoted
by $\bb E_{\mu_n}^\beta$. To simplify notation, we will denote  $\bb P_{\nu_\rho}^\beta$ by  $\bb P_\rho^\beta$ and we do not index $\bb P_{\mu_n}^\beta$ nor
 $\bb E_{\mu_n}^\beta$ in $\alpha$.
 We define also $\chi(\rho):=\rho(1-\rho)$, the so-called \emph{static compressibility} of the system.

\subsection{The Operators $\A$ and $\B$} \label{operator A beta}
We introduce some spaces we will use in the sequel.

\begin{definition}
Let $L^2_\beta(\mathbb{R})$ be the space of functions $H:\mathbb{R}\rightarrow{\mathbb{R}}$ with $\|H\|_{2,\beta}<+\infty$, where
\begin{equation*}
\|H\|_{2,\beta}^2=\left\{\begin{array}{cl}
 \int_{\mathbb{R}}H(u)^2 du, & \mbox{if}\,\,\, \beta\neq{1}\\

 \\
 \int_{\mathbb{R}}H(u)^2 du+H(0)^2, & \mbox{if}\,\,\, \beta=1.
\end{array}
\right.
\end{equation*}
\end{definition}
Notice that, for $\beta\neq{1}$, the norm $\|\cdot\|_{2,\beta}$ is the usual $L^2(\bb R)$-norm with respect to the Lebesgue measure that we denote by
$\lambda$. For simplicity in this case we write $\|\cdot\|_{2}$.
For $\beta=1$, the norm $\|\cdot\|_{2,\beta}$ is the $L^2(\bb R)$-norm with respect to the
measure $\lambda+\delta_{0}$, where $\delta_{u}$ denotes the Dirac measure at
the point $u\in{\bb R}$.

In the sequel, given $H:\mathbb{R}\rightarrow{\mathbb{R}}$, we  denote
$$H(0^+):=\displaystyle\lim_{\at{u\to 0,}{u>0}}H(u)\quad \textrm{ and }\quad  H(0^-):=\displaystyle\lim_{\at{u\to
0,}{u<0}}H(u)\,,$$
when the above limits exist.
For $k\in{\mathbb{N}}$, we denote by $H^{(k)}(x)$, the
$k^{\textrm{th}}$-derivative of $H$ at the point $x\in{\mathbb{R}}$. For $k=0$,  $H^{(0)}(x)$ means $H(x)$.

\begin{definition}
Define $\mc S(\bb R\backslash \{0\})$ as the space of functions $H:\bb R\to\bb R$ such that $H\in C^\infty(\bb R\backslash\{0\})$, $H$ is continuous from the right at $x=0$ and $H$ satisfies
\begin{equation*}
 \Vert H \Vert_{k,\ell}\;:=\;\sup_{x\in \bb R\backslash{\{0\}}}|(1+|x|^\ell)
\,H^{(k)}(x)|\;<\;\infty\,,
\end{equation*}
\noindent for all integers $k,\ell\geq 0$ and $H^{(k)}(0^-)=H^{(k)}(0^+)$, for all $k$ integer, $k\geq 1$.
\end{definition}

Next, we present the domains for  $\A$ and $\B$.
\begin{definition} \label{S heatperiodic}
For $\beta\in[0,1)$, we define $\mc S_\beta(\bb R)$ as the subset of $\mc S(\bb R\backslash \{0\})$
composed of functions $H$ satisfying
$$H(0^-)=H(0^+)\,.$$
\end{definition}
Notice that the space above is nothing more than the usual Schwartz space $\mc S(\bb R)$. Fix now
$\alpha>0$.
\begin{definition}  \label{S heatrobin}
For $\beta=1$,  we define $\mc S_\beta(\bb R)$ as the subset of $\mc S(\bb R\backslash \{0\})$
composed of functions $H$ satisfying
$$H^{(1)}(0^+)\;=\; H^{(1)}(0^-)\;=\; \alpha \Big( H(0^+)-H(0^-)\Big)\,.$$
\end{definition}
\begin{definition}  \label{S heatneumann}
For $\beta\in(1,+\infty]$,  we define $\mc S_\beta(\bb R)$ as the subset of $\mc S(\bb R\backslash \{0\})$
composed of functions $H$ satisfying
$$H^{(1)}(0^+)\;=\; H^{(1)}(0^-)\;=\; 0\,.$$
\end{definition}
\begin{proposition}\label{frechet}
 For any chosen $\beta\in[0,+\infty]$, the space $\mc S_\beta(\bb R)$ is a Fr\'echet space.
\end{proposition}
The definition of a Fr\'echet space can be found, for instance, in \cite{rs}. The proof that $\mc S(\bb R\backslash \{0\})$
is a Fr\'echet space follows the same lines of that of \cite{rs} for the usual Schwartz space $\mc S(\bb R)$, and for that reason it will be omitted.
Since the spaces $\mc S_\beta(\bb R)$ are closed vector spaces of $\mc S(\bb R\backslash\{0\})$, this implies they are also Fr\'echet
spaces.
We notice that along the paper we only use this fact when we invoke the result of \cite{Mit} about tightness of stochastic process
taking values in Fr\'echet spaces.

\begin{definition}
We define the operators  $\A: \mc S_\beta(\bb R)\rightarrow \mc S(\bb R)$  and
$\B: \mc S_\beta(\bb R)\rightarrow \mc S(\bb R)$ by
\begin{equation*}
\B H(u)\;=\;\left\{\begin{array}{cl}
H^{(1)}(u), &  \mbox{if}\,\,\,\,u\neq 0\,,\\
H^{(1)}(0^+), &\mbox{if}\,\,\,\,u=0\,,
\end{array}
\right.
\qquad \textrm{and}\qquad
\A H(u)\;=\;\left\{\begin{array}{cl}
H^{(2)}(u), &  \mbox{if}\,\,\,\,u\neq 0\,,\\
H^{(2)}(0^+), &\mbox{if}\,\,\,\,u=0\,.
\end{array}
\right.
\end{equation*}
\end{definition}
Notice that the operators $\B$ and $\A$ are essentially the usual derivative and the usual second derivative, but defined in specific
domains.

\subsection{Hydrodynamic limit, PDE's and semigroups}
The hydrodynamic limit for the exclusion process with a slow bond was already studied in
\cite{fgn,fgn2}. We state it here for completeness.
Let $g: \mathbb R \to [0,1]$ be a piecewise continuous function and suppose that there exists a constant $C_g$ such that $g-C_g$ has compact support. Let $n \in \mathbb N$ be a scaling
parameter. We define a probability measure $\mu^n$ in $\Omega$  by
\[
\mu^n\big(\eta(z_1)=1,...,\eta(z_\ell)=1\big) = \prod_{i=1}^\ell g(z_i/n),
\]
for any set $\{z_1,...,z_\ell\} \subseteq \mathbb Z$ and $\ell\in{\bb N}$.
Let  $\{\eta_{tn^2};t  \geq 0\}$ have initial distribution $\mu^n$.
We define the {\em empirical measure} $\{\pi_t^n; t \geq 0\}$ as the measure-valued process given by
\[
\pi_t^n(dx) = \frac{1}{n} \sum_{x \in \mathbb Z} \eta_{tn^2}(x) \delta_{\frac{x}{n}}(dx).
\]
In words, the empirical measure represents the time evolution of the spatial density of particles.

Now, let  $\mc M_+$ be the space of positive measures on $\bb R$ with total
mass bounded by one, endowed with the weak topology.

\begin{theorem}[Franco, Gon\c calves, Neumann \cite{fgn,fgn2}] \label{p1}

\quad

For any $\beta\in{[0,+\infty]}$ and for any $T \geq 0$, as $n\to+\infty$, the sequence of measure valued processes $\{\pi_t^n(dx);\,$ $t \in [0,T]\}_{n \in \mathbb N}$
converges in probability with respect to the Skorohod topology of  $\mc D([0,T], \mathcal M_+(\mathbb R))$, to some $\{u(t,x)dx;\, t \in [0,T]\}$. Moreover,

\begin{itemize}
\item
for $\beta\in [0,1)$, $\{u(t,x); t \geq 0,\, x \in \mathbb R\}$ is the unique weak
solution of the heat equation
\begin{equation}\label{pde1}
\left\{
\begin{array}{ll}
\partial_t u(t,x) = \; \partial^2_{xx} u(t,x), &t \geq 0,\, x \in \mathbb R\\
u(0,x) = \; g(x), &x \in \mathbb R.
\end{array}
\right.
\end{equation}
\item
for $\beta=1$, $\{u(t,x); t \geq 0, x \in \mathbb R\}$ is the
unique weak solution of the heat equation with a boundary condition of Robin's type at $x=0$
\begin{equation}\label{pde2}
\left\{
\begin{array}{ll}
 \partial_t u(t,x) = \; \partial^2_{xx} u(t,x), &t \geq 0,\, x \in \mathbb R\backslash\{0\}\\
\p_x u(t,0^+)=\p_x u(t,0^-)=\alpha\{u(t,0^+)-u(t,0^-) \},  &t \geq 0\\
 u(0,x) = \; g(x), &x \in \mathbb R.
\end{array}
\right.
\end{equation}
\item
for $\beta\in (1,+\infty]$, $\{u(t,x); t \geq 0, x \in \mathbb R\}$ is the unique weak
solution of the heat equation with a boundary condition of Neumann's type at $x=0$
\begin{equation}\label{pde3}
\left\{
\begin{array}{ll}
\partial_t u(t,x) = \; \partial^2_{xx} u(t,x), &t \geq 0,\, x \in \mathbb R\backslash\{0\}\\
\p_x u(t,0^+)=\p_x u(t,0^-)=0,  &t \geq 0\\
u(0,x) = \; g(x), &x \in \mathbb R.
\end{array}
\right.
\end{equation}
\end{itemize}
\end{theorem}
The previous theorem corresponds to Theorem 4.1 of \cite{fgn} and Theorem 2.1 of \cite{fgn2}, considering the process evolving in finite volume (periodic). However, the proof for infinite volume is the same,
aside from some  topological adaptations. The definitions of weak solutions of equation \eqref{pde1}, \eqref{pde2} and \eqref{pde3} are the same as given in \cite{fgn2} for the finite volume case, with the additional usual assumption that the test functions are compactly supported.

Each one of the partial differential equations mentioned above is linear. As we will
see later, in order to prove the existence of a Ornstein-Uhlenbeck process with characteristics $\Delta_\beta$ and $\nabla_\beta$, we will make use of the explicit expression for the semigroups corresponding to $\Delta_\beta$. The
semigroup of
\eqref{pde1} is classical and it acts on $g\in\mc S_\beta(\bb R)$ with $\beta\in[0,1)$ given in Definition \ref{S heatperiodic}, as
  \begin{equation}\label{sem heat eq}
  T_t g(x)\;=\; \frac{1}{\sqrt{4\pi t}}\int_{\bb R} e^{-\frac{(x-y)^2}{4t}}g(y)\,dy\,,\quad \textrm{for }x\in\bb R\,.
 \end{equation}
The semigroup of \eqref{pde3} is also known and it acts on $g\in\mc S_\beta(\bb R)$ with  $\beta\in(1,+\infty]$ given in Definition \ref{S heatneumann}, as
  \begin{equation}\label{sem heat eq neu}
  T_t^\textrm{Neu} g(x)\;=\;
  \begin{cases}
\displaystyle   \frac{1}{\sqrt{4\pi t}}\int_{0}^{+\infty}\Big[
e^{-\frac{(x-y)^2}{4t}}+e^{-\frac{(x+y)^2}{4t}}\Big]g(y)\,dy\,,\quad &\textrm{for }x>0\,,\\
\displaystyle  \frac{1}{\sqrt{4\pi t}}\int_{0}^{+\infty}\Big[
e^{-\frac{(x-y)^2}{4t}}+e^{-\frac{(x+y)^2}{4t}}\Big]g(-y)\,dy\,,\quad &\textrm{for }x<0\,.\\
\end{cases}
 \end{equation}

Denote by $g_{\textrm{even}}$ and $g_{\textrm{odd}}$ the even and odd parts of a function $g:\bb R\to \bb R$, respectively, or
else, for $x\in{\mathbb{R}}$,
\begin{equation*}
 g_{\textrm{even}}(x)=\frac{g(x)+g(-x)}{2}\quad \textrm{and} \quad g_{\textrm{odd}}(x)=\frac{g(x)-g(-x)}{2}\,.
\end{equation*}
\begin{proposition}\label{prop23}
 The semigroup of \eqref{pde2} acts on $g\in\mc S_\beta(\bb R)$ with $\beta=1$ given in Definition \ref{S heatrobin}, as
  \begin{equation*}
  \begin{split}
  & T_t^\alpha g(x)= \frac{1}{\sqrt{4\pi t}}\Bigg\{\int_{\bb R}
e^{-\frac{(x-y)^2}{4t}} g_{\textrm{{\rm even}}}(y)\,dy \\
    & + e^{2\alpha x}\int_x^{+\infty} e^{-2\alpha z} \int_0^{+\infty}
\Big[(\pfrac{z-y+4\alpha t}{2t})e^{-\frac{(z-y)^2}{4t}}+(\pfrac{z+y-4\alpha t}{2t})e^{-\frac{(z+y)^2}{4t}}\Big]\,
g_{\textrm{{\rm odd}}}(y)\, dy\, dz\,\Bigg\}\,,\\
  \end{split}
  \end{equation*}
\noindent for $x>0$ and
  \begin{equation*}
  \begin{split}
 & T_t^{\alpha} g(x)= \frac{1}{\sqrt{4\pi t}}\Bigg\{\int_{\bb R}
e^{-\frac{(x-y)^2}{4t}} g_{\textrm{{\rm even}}}(y)\,dy \\
    & - e^{-2\alpha x}\int_{-x}^{+\infty} e^{-2\alpha z} \int_0^{+\infty}
\Big[(\pfrac{z-y+4\alpha t}{2t})e^{-\frac{(z-y)^2}{4t}}+(\pfrac{z+y-4\alpha t}{2t})e^{-\frac{(z+y)^2}{4t}}\Big]\,
g_{\textrm{{\rm odd}}}(y)\, dy\, dz\,\Bigg\}\,,\\
  \end{split}
  \end{equation*}
\noindent for $x<0$.
\end{proposition}

Throughout the text we will simply write $T^\beta_t$ for the three semigroups $T_t$, $T^\alpha_t$ and
$T_t^{\textrm{Neu}}$, corresponding to the regimes $\beta\in[0,1)$, $\beta=1$ and $\beta\in(1,{+\infty}]$, respectively.

\begin{remark}
Since a smooth solution is a weak solution, the classical formulas \eqref{sem heat eq} and \eqref{sem heat eq neu} and the previous proposition  guarantee that the weak solutions in Theorem \ref{p1} are, indeed, smooth solutions.
\end{remark}

 Notice that $T_t^{\textrm{Neu}}$ evolves a function in independent ways in each half line, but $T^\alpha_t$ does not.
From this characterization of the semigroup $T_t^{\alpha}$, we get almost for free the following result:
\begin{proposition}\label{conver_pde}
 Let $u,u^\alpha,u^{\textrm{Neu}}:\bb R_+\times \bb R\to [0,1]$ be the unique smooth solution  of \eqref{pde1}, \eqref{pde2} and \eqref{pde3}, respectively. Then,
\begin{equation*}
 \lim_{\alpha\to {+\infty}} u^\alpha(t,x)\;=\;u(t,x) \quad \textrm{ and } \quad   \lim_{\alpha\to 0}
u^\alpha(x,t)\;=\;u^{\textrm{Neu}}(t,x)\,,
 \end{equation*}
 for all $(t,x)\in \bb R_+\times (\bb R\backslash\{0\})$.
Besides that, for fixed $t>0$, the following convergence holds
\begin{equation*}
 \lim_{\alpha\to {+\infty}} \Vert u^\alpha(t,\cdot)-u(t,\cdot)\Vert_{L^p(\bb R)}\;=\;0 \quad \textrm{ and } \quad   \lim_{\alpha\to
0}
\Vert u^\alpha(t,\cdot)-u^{\textrm{Neu}}(t,\cdot)\Vert_{L^p(\bb R)}\;=\;0,
\end{equation*}
for all $p\in[1,{+\infty}]$.
 \end{proposition}
 The convergence above can be improved to some extent related to space and time simultaneously. Since this is not the main
issue of this paper, we do not enter into details on this.

\subsection{Ornstein-Uhlenbeck process}
Based on \cite{HS, kl}, we give here a characterization of the generalized Ornstein-Uhlenbeck process which is a  solution of
\begin{equation}\label{eq Ou}
d\mathcal{Y}_t=\A \mathcal{Y}_tdt+\sqrt{2\chi(\rho)}\B d\mc{W}_t\,,
\end{equation}
 where $\mc{W}_t$ is a space-time white noise of unit variance, in terms of a martingale problem. We will see later that this process governs the equilibrium fluctuations of the density of particles.  In spite of having a dependence of  $\mathcal{Y}_t$ on $\beta$, in order to keep notation simple, we do not index on it.

In what follows $\mathcal{S}'_{\beta}(\mathbb{R})$ denotes the space of bounded linear functionals $f:
\mathcal{S}_{\beta}(\mathbb{R})\rightarrow{\bb R}$ and  $\mathcal{D}([0,T],\mathcal{S}'_{\beta}(\mathbb{R}))$
(resp. $\mathcal{C}([0,T],\mathcal{S}'_{\beta}(\mathbb{R}))$)
 is the space of  c\`adl\`ag (resp. continuous) $\mathcal{S}'_{\beta}(\mathbb{R})$ valued functions endowed with the Skohorod topology.

\begin{proposition}\label{pp1}
There exists an unique random element $\mc Y$ taking values in the space $\mc C([0,T],\mathcal{S}'_{\beta}(\bb R))$ such
that:
\begin{itemize}
\item[i)] For every function $H \in \mathcal{S}_{\beta}(\bb R)$, $\mc M_t(H)$ and $\mc N_t(H)$ given by
\begin{equation}\label{lf1}
\begin{split}
&\mc M_t(H)= \mc Y_t(H) -\mc Y_0(H) -  \int_0^t \mc Y_s(\A H)ds\,,\\
&\mc N_t(H)=\big(\mc M_t(H)\big)^2 - 2\chi(\rho) \; t\,\|\B H\|_{2,\beta}^2
\end{split}
\end{equation}
are $\mc F_t$-martingales, where for each $t\in{[0,T]}$, $\mc F_t:=\sigma(\mc Y_s(H); s\leq t,  H \in \mathcal{S}_{\beta}(\bb R))$.

\item[ii)] $\mc Y_0$ is a gaussian field of mean zero and covariance given on $G,H\in{\mathcal{S}_{\beta}(\mathbb{R})}$ by
\begin{equation}\label{eq:covar1}
\mathbb{E}_\rho^\beta\big[ \mc Y_0(G) \mc Y_0(H)\big] =  \chi(\rho)\int_{\mathbb{R}} G(u) H(u) du\,.
\end{equation}
\end{itemize}
Moreover, for each $H\in\mc S_\beta(\bb R)$, the stochastic process $\{\Y_t(H)\,;\,t\geq 0\}$ is  gaussian , being the
distribution of $\Y_t(H)$  conditionally to
$\mc F_s$, for $s<t$, normal  of mean $\Y_s(T^\beta_{t-s}H)$ and variance $\int_0^{t-s}\Vert \B T^\beta_{r}
H\Vert^2_{2,\beta}\,dr$.
\end{proposition}

 We call the  random element $\mc Y_\cdot$ the generalized Ornstein-Uhlenbeck process of
characteristics $\A$ and $\B$.  From the second equation in \eqref{lf1} and Levy's Theorem on the martingale characterization of Brownian motion, the process
\begin{equation}\label{Bmotion}
 \mc M_t(H)(2\chi(\rho)\|\B H\|_{2,\beta}^2)^{-1/2}
\end{equation}
is a standard Brownian motion. Therefore, in view of  Proposition \ref{pp1}, it makes sense to say that $\Y$ is
the formal solution of \eqref{eq Ou}.

\subsection{Equilibrium Density Fluctuations}
In order to establish the C.L.T. for the empirical measure under the invariant state $\nu_{\rho}$,
we need to introduce the density fluctuation field as the linear functional acting on test functions $H$ as
\begin{equation*}\label{flut}
 \mc Y^n_{t}(H)= \frac{1}{\sqrt n}\sum_{x\in{\mathbb{Z}}}H\Big(\frac{x}{n}\Big)(\eta_{tn^2}(x)-\rho).
 \end{equation*}

We are in position to state the fluctuations for the density of particles.
\begin{theorem}[C.L.T. for the density of particles]\label{flu1}

\quad

Consider the Markov process $\{\eta_{tn^2}: t\geq{0}\}$ starting from the invariant state $\nu_\rho$.
Then, the sequence of processes $\{\mathcal{Y}_{t}^n\}_{ n\in{\bb N}}$ converges in distribution, as $n\rightarrow{+\infty}$, with respect to the
Skorohod topology
of $\mathcal{D}([0,T],\mathcal{S}'_{\beta}(\bb R))$ to   $\mathcal{Y}_t$
in $\mathcal{C}([0,T],\mathcal{S}'_{\beta}(\bb R))$, the generalized Ornstein-Uhlenbeck process of characteristics $\Delta_\beta,\nabla_\beta$ which is the formal solution of the equation
\begin{equation} \label{OU}
d\mathcal{Y}_t=\A \mathcal{Y}_tdt+\sqrt{2\chi(\rho)} \B d\mc{W}_t,
\end{equation}
 where $\mc{W}_t$ is a space-time white noise of unit variance and the operators $\A$ and $\B$ were defined in Subsection
\ref{operator A beta}.
\end{theorem}

\subsection{Equilibrium Current Fluctuations}\label{sub_eq}
Next, we introduce the notion of current of particles through a fixed bond for our microscopic dynamics of generator
$\mc L_n$ evolving on the diffusive time scale $tn^2$  and starting from the invariant
state $\nu_{\rho}$.

 For a site $x\in{\bb Z}$, denote
by ${J}^n_{x,x+1}(t)$ the current of particles over the bond $\{x,x+1\}$, which is the total number of jumps from the
site
$x$ to the site $x+1$ minus
the total number of jumps from the site $x+1$ to the site $x$ in the time interval $[0,tn^2]$.

 Let $u\in{\mathbb{R}}$ be a macroscopical point, to which we associate in the microscopical lattice the bond of vertices
$\{\lfloor un \rfloor -1, \lfloor un \rfloor\}$. Here  $\lfloor un\rfloor$ denotes  the biggest integer not larger than $un$. To simplify notation, we will simply write
\begin{equation*}
{J}^n_{u}(t):= {J}^n_{\lfloor un \rfloor -1,\lfloor un \rfloor}(t)\,.
\end{equation*}
Now, we state the C.L.T. for the current. For that purpose we need to introduce some notation.
Denote by $\Phi_{2t}(\cdot)$  the tail of the distribution function of a gaussian random variable with
mean zero and variance $2t$, that is, for $x\in{\mathbb{R}}$,
\begin{equation*}
\Phi_{2t}(x):=\int^{{+\infty}}_x\frac{e^{-u^2/{4t}}}{\sqrt{4\pi t}}du \,.
\end{equation*}

\begin{theorem}[C.L.T. for the current of particles] \label{current clt}

\quad

Under $\bb P_{\rho}^\beta$, for every $t\geq{0}$ and every $u\in\bb R$,
\begin{equation*}
 \frac{{J}^n_u(t)}{\sqrt n}\xrightarrow[{n\rightarrow{+\infty}}]\,{J}_u({t})
\end{equation*}
in the sense of finite-dimensional distributions, where ${J}_u({t})$  is gaussian with covariances given by
\medskip

$\bullet$ for $\beta\in{[0,1)}$,
\begin{equation}\label{fBM_cov}
\mathbb{E}_\rho^\beta[{J}_u(t){J}_u(s)]=\chi(\rho)\Big(\sqrt{\frac{t}{\pi}}+\sqrt{\frac{s}{\pi}}-\sqrt{\frac{t-s}{
\pi}}\Big)\,,
\end{equation}
that is $J_u(t)$ is a fractional Brownian motion of Hurst exponent $1/4$.

\quad

$\bullet$ for $\beta=1$,
\begin{equation*}
\begin{split}
\mathbb{E}_\rho^\beta[{J}_u(t){J}_u(s)]=\chi(\rho)\Big(&\sqrt{\frac{t}{\pi}}+\frac{\Phi_{2t}(2u+4\alpha
t)\,e^{4\alpha u+4\alpha^2t}}{2\alpha}-\frac{\Phi_{2t}(2u)}{2\alpha}\\
+&\sqrt{\frac{s}{\pi}}+\frac{\Phi_{2s}(2u+4\alpha
s)\, e^{4\alpha u+4\alpha^2s}}{2\alpha}-\frac{\Phi_{2s}(2u)}{2\alpha}\\
-&\sqrt{\frac{t-s}{\pi}}-\frac{\Phi_{2(t-s)}(2u+4\alpha (t-s))\,e^{4\alpha
u+4\alpha^2(t-s)}}{2\alpha}+\frac{\Phi_{2(t-s)}(2u)}{2\alpha}\Big)\,.
\end{split}
\end{equation*}

\quad

$\bullet$ for  $\beta\in (1,{+\infty}]$,
\begin{equation}\label{neumann covariance}
\begin{split}
\mathbb{E}_\rho^\beta[{J}_u(t){J}_u(s)]=\chi(\rho)\Big(&\sqrt{\frac{t}{\pi}}\Big[1-e^{-u^2/t}\Big]+2u\,\Phi_{2t}(2u)\\+&\sqrt{\frac{s}{\pi}}\Big[1-e^{
-u^2/s}\Big]+2u\,\Phi_{2s}(2u)\\
-&\sqrt{\frac{t-s}{\pi}}\Big[1-e^{-u^2/(t-s)}\Big]-2u\,\Phi_{2(t-s)}(2u)\Big).
\end{split}
\end{equation}

\end{theorem}
It is of particular interest the covariance at $u=0$, corresponding to the current through the slow bond $\{-1,0\}$. If $\beta\in [0,1)$,
the covariance corresponds to the one of a fractional Brownian motion of Hurst exponent $1/4$. If $\beta\in (1, {+\infty}]$, the
covariance equals zero as expected, since the Neumann's boundary conditions at $x=0$ make of it an isolated boundary. Finally, for $\beta=1$, we
obtain a family, indexed in the parameter $\alpha$, of gaussian processes interpolating the
fractional Brownian motion of parameter $1/4$ and the degenerate process identically equal to zero. Such interpolation is made
clear in the next  corollary. Before its statement, we emphasize that at the critical value $\beta=1$, the limit of $J^n_u(t)/\sqrt n$ does depend
on $\alpha$. Let us denote it by
${J}^\alpha_u({t}).$
\begin{corollary} \label{limit robin current}

For every $t\geq{0}$ and every $u\in\bb R$,

\begin{equation*}
{J}^\alpha_u({t})\xrightarrow[{\alpha\rightarrow{+\infty}}]\,{J}_u({t})\,,
\end{equation*}
 where $J_u(t)$ is the fractional  Brownian motion
with Hurst exponent $1/4$
and
\begin{equation*}
{J}^\alpha_u({t})\xrightarrow[{\alpha\rightarrow{0}}]\,{J}_u({t})\,,
\end{equation*}
where $J_u(t)$ is the   gaussian process with covariances given by  \eqref{neumann
covariance}. The convergence is in the sense of finite dimensional distributions.
 \end{corollary}

\subsection{Fluctuations of a tagged particle}
As a consequence of last construction, we are able to deduce the behavior of a single tagged particle as done in \cite{g,jl}. For that purpose, fix $\rho\in{(0,1)}$, $u>0$ and consider $\eta_{tn^2}$ starting from the measure $\nu_{\rho}$ conditioned to have a particle at the site $\lfloor un\rfloor$, that we denote by $\nu_\rho^u$. More precisely, $\nu_{\rho}^u(\cdot):=\nu_\rho(\,\cdot\,|\eta_{tn^2}(\lfloor un\rfloor) =1)$. We notice that from symmetry arguments, the same reasoning holds for $u<0$.
 We couple the system starting from $\nu_\rho^u$ and starting from $\nu_\rho$, in such a way that both
processes differ at most in one site at any given time. Then, the analogue of the results stated in Theorems \ref{flu1} and \ref{current clt} for the starting measure  $\nu_\rho^u$ follow from those results where the system is taken starting from $\nu_\rho$.

Let $X^n_u(t)$ denote the position at time $tn^2$ of a tagged particle initially at the site $\lfloor un\rfloor$.
Since we are in dimension one, the order between particles is preserved and as a consequence, for all $k\geq 1$,
\begin{equation}\label{relation tp cur dp}
\{X^n_u(t)\geq{k}\}=\Big\{ {J}^n_{u}(t)\geq{\sum_{x=\lfloor un \rfloor}^{\lfloor un \rfloor+k-1}\eta_{tn^2}(x)}\Big\}.
\end{equation}
 Last relation together with Theorem \ref{current clt} gives us
\begin{theorem}[C.L.T. for a tagged particle] \label{tagged clt}

\quad

Under $\bb{P}^\beta_{\nu^u_\rho}$, for all $\beta\in{[0,+\infty]}$,  every $u\in\bb R$ and $t\geq{0}$
\begin{equation*}
\frac{{X}^n_u(t)}{\sqrt{n}}\xrightarrow[{t\rightarrow{+\infty}}]\,X_u(t)
\end{equation*}
in the sense of finite-dimensional distributions, where $X_u(t)=J_u(t)/\rho$ in law and $J_u(t)$ is the same as in Theorem
\ref{current clt}. In particular, the covariances of the process $X_u(t)$ are given by
$\mathbb{E}_\rho^\beta[X_u(t)X_u(s)]=\rho^{-2}\,\mathbb{E}_\rho^\beta[{J}_u(t){J}_u(s)]$.
\end{theorem}

We do not present the proof of this theorem since it is very similar to the one presented in \cite{g,jl}. We only remark that in this case the mean of the current and the tagged particle is zero since the dynamics is symmetric. For tightness issues we refer the reader to \cite{ps}, in which the case $\beta=0$  and $\alpha=1$ was considered.

We observe that in the case $\beta\in{(1,{+\infty}]}$, the tagged particle  starting at the origin moves microscopically but we do not see its fluctuations macroscopically, since the variance of $X_0(t)$ equals zero.

Here we describe the paper structure: Propositions \ref{prop23} and \ref{conver_pde} are proved in Section
\ref{s4}. We remark that Proposition \ref{conver_pde} is an extra
result, which is not applied along the text. Proposition \ref{prop23}
is invoked in Sections \ref{s5} and \ref{s6}. Propositions \ref{pp1} and Theorem \ref{flu1}  are proved in Section \ref{s3}
and Section \ref{s5} in the following way. Existence of the
Ornstein-Uhlenbeck process and convergence of the density fluctuation field along subsequences is proved in  Section \ref{s3}, while the uniqueness of the Ornstein-Uhlenbeck
process is reserved to Section \ref{s5}. Theorem \ref{current clt}, Corollary \ref{limit robin current}
and Theorem \ref{tagged clt} are proved in Section \ref{s6}. Finally, in Section \ref{s7}, we present $L^2$-estimates that are
invoked in Section \ref{s3} and Section \ref{s6}.

\section{Central Limit Theorem for the density of particles}\label{s3}

In this section we prove Theorem \ref{flu1}.
As usual in convergence of stochastic process, there are two facts
to be shown: convergence of finite-dimensional distributions of
$\mathcal{Y}_{t}^{n}$ to those of $\mathcal{Y}_t$ and tightness of the sequence $\{\mathcal{Y}_{t}^{n}\}_{n\in{\mathbb{N}}}$. We
start by the former.

\subsection{Characterization of limit points}

In this section we want to prove that the limit points of the sequence $\{\mathcal{Y}_t^n\}_{n\in{\mathbb{N}}}$ satisfy
Proposition \ref{pp1}. We start by showing that any limit point of the sequence $\{\mathcal{Y}_t^n\}_{n\in{\mathbb{N}}}$ solves \eqref{lf1}.

\subsubsection{Martingale problem}

By Dynkin's formula, for a given function $H\in{\mc {S}_\beta(\bb R)}$,
\begin{equation*}\label{martingal}
\mc M^n_t(H)=\mc Y^n_{t}(H)- \mc Y^n_0(H)-\int_{0}^{t}n^{2}\mc L_{n}\,\mc Y^n_s(H)\,ds
\end{equation*}
is a martingale with respect to the natural filtration $\mathcal{G}_{t}^n=\sigma(\eta_{sn^2},s\leq{t})$.
Doing simple computations we get
\begin{equation}\label{martingale decomposition}
\mc M^n_t(H)=\mc Y^n_{t}(H)- \mc Y^n_0(H)- \mc I_{t}^n(H),
\end{equation}
where
\begin{equation}\label{I}
 \mc I_{t}^{n}(H)=\int_{0}^{t}\frac{1}{\sqrt{n}}\sum_{x\in{\bb
Z}}n^2\mathbb{L}_{n}H\big(\pfrac{x}{n}\big)\eta_{sn^2}(x)\,ds\,
\end{equation}
and
$\mathbb{L}_n$ is the generator of the random walk on $ \bb Z$ given on $H:\bb Z \rightarrow \mathbb{R}$ and $x \in \bb Z$ by
\begin{equation*}
(\mathbb{L}_n H)(\pfrac{x}{n}) =  \xi^n_{x,x+1} \, \big[ H(\pfrac{x+1}{n})
- H(\pfrac{x}{n}) \big] + \xi^n_{x-1,x} \, \big[H(\pfrac{x-1}{n}) - H(\pfrac{x}{n}) \big] .
\end{equation*}
Note that, despite we do not index, the operator $\mathbb{L}_n$ depends on $\beta$.

We take in particular
$H\in{\mc{S}_\beta(\bb R)}$.
By the fact that the sum $\sum_{x\in{\bb Z}}n^2\mathbb{L}_{n}H(\pfrac{x}{n})$ is null
 and by adding and subtracting $\int_{0}^{t}\mc Y^n_s(\A H)\,ds$ to $\mc I_{t}^{n}(H)$, we can rewrite the martingale
$ \mc M^n_t(H)$ as
\begin{equation*}
 \mc M^n_t(H)=\mc Y^n_{t}(H)- \mc Y^n_0(H)-\int_{0}^{t}\mc Y^n_s(\A  H)\,ds-R_t^{n,\beta}(H),
\end{equation*}
where
\begin{equation*}
R_t^{n,\beta}(H):=\int_{0}^{t}\frac{1}{\sqrt{n}}\sum_{x\in\bb Z}\Big\{ n^2\mathbb{L}_{n}H\big(\pfrac{x}{n}\big)-(\A
H)(\pfrac{x}{n})
\Big\}\bar{\eta}_{sn^2}(x)\,ds\,
\end{equation*}
and for each $x\in{\mathbb{Z}}$, the centered random variable $\bar{\eta}_{sn^2}(x)$ denotes  $\eta_{sn^2}(x)-\rho$.

In some
points ahead we will write $\pfrac{0}{n}$ as zero to emphasize the discretization of space and make easier to follow the computations.

We start by showing that $R_t^{n,\beta}(H)$ is negligible in $L^2(\bb P_\rho^\beta)$, for all $H\in{\mc {S}_\beta(\bb R)}$.
 \begin{proposition}
For every $t\in{[0,T]}$, $\beta\in[0,{+\infty}]$ and $H\in{\mc {S}_\beta(\bb R)}$,
\begin{equation*}
\lim_{n\rightarrow{+\infty}}\mathbb{E}_\rho^\beta\Big[\big( R_t^{n,\beta}(H)\big)^2\Big]=0.
\end{equation*}
\end{proposition}
\begin{proof}
 Separating the sites close to the slow bond, we can rewrite
\begin{equation}
\begin{split}\label{eq3.4}
 R_t^{n,\beta}(H)
=&\int_{0}^{t}\frac{1}{\sqrt{n}}\sum_{x\neq -1, 0}\Big\{ n^2\mathbb{L}_{n}H\big(\pfrac{x}{n}\big)
-(\A  H)(\pfrac{x}{n})\Big\}\,\bar{\eta}_{sn^2}(x)\,ds\\
+&\int_{0}^{t}\frac{1}{\sqrt{n}}\Big\{ n^2\mathbb{L}_{n}H\big(\pfrac{-1}{n}\big)
-(\A  H)\big(\pfrac{-1}{n}\big)\Big\}\,\bar{\eta}_{sn^2}(-1)\,ds\\
+&\int_{0}^{t}\frac{1}{\sqrt{n}}\Big\{n^2\mathbb{L}_{n}H\big(\pfrac{0}{n}\big)
- (\A  H)\big(\pfrac{0}{n}\big)\Big\}\,\bar{\eta}_{sn^2}(0)\,ds.\\
\end{split}
\end{equation}
The operator $\A$ distinguishes of the usual laplacian operator essentially in the domain. Outside of the macroscopic point
$0$, for any $\beta$, the operator $\A$ behaves as the usual laplacian. Besides that, for $x\neq -1,0$,  the  term
$n^2\bb L_n(x)$ is exactly the discrete laplacian.
Hence, by the classical approximation of the continuous laplacian by the discrete laplacian, the first integral in \eqref{eq3.4}
is $O(1/\sqrt{n})$.

Since $\A  H$ is bounded, in order to show that
the sum of the second and third integrals in \eqref{eq3.4} goes to zero, it is enough to show that
\begin{equation*}
\begin{split}
r_t^{n,\beta}
:=&\int_{0}^{t}\frac{1}{\sqrt{n}}\Big\{ n^2\mathbb{L}_{n}H\big(\pfrac{-1}{n}\big)
\Big\}\,\bar{\eta}_{sn^2}(-1)\,ds+\int_{0}^{t}\frac{1}{\sqrt{n}}\Big\{n^2\mathbb{L}_{n}H\big(\pfrac{0}{n}\big)
\Big\}\,\bar{\eta}_{sn^2}(0)\,ds\\
\end{split}
\end{equation*}
goes to zero, as  $n\to{+\infty}$. Recalling the definition of $\bb L_n$ we arrive at
\begin{equation}\label{eq3.5}
\begin{split}
 r_t^{n,\beta}&=\int_{0}^{t}\!\frac{1}{\sqrt{n}}\Big\{ \alpha n^{2-\beta}\big[H\big(\pfrac{0}{n}\big)-H\big(\pfrac{-1}{n}\big)\big]
-n^2\big[H\big(\pfrac{-1}{n}\big)-H\big(\pfrac{-2}{n}\big)\big]\Big\}\,\bar{\eta}_{sn^2}(-1)\,ds\\
&+\int_{0}^{t}\!\frac{1}{\sqrt{n}}\Big\{n^2\big[H\big(\pfrac{1}{n}\big)-H\big(\pfrac{0}{n}\big)\big]
-\alpha n^{2-\beta}\big[H\big(\pfrac{0}{n}\big)-H\big(\pfrac{-1}{n}\big)\big]\Big\}\bar{\eta}_{sn^2}(0)\,ds.
\end{split}
\end{equation}
For each regime of $\beta$, namely, $\beta\in [0,1)$, $\beta=1$ and $\beta\in(1,{+\infty}]$, we present a specific argument to show
that $r_t^{n,\beta}$ vanishes in $L^2(\bb P_\rho^\beta)$,
as $n\to{+\infty}$.
Let us begin with the

\quad

\noindent $\bullet$ Case $\beta\in[0,1)$:

\smallskip

Recall that in this case $\mc S_\beta(\bb R)=\mc S(\bb R)$ and thus $H$ is smooth. Let
\begin{equation*}
(\Delta_n H)\big(\pfrac{x}{n}\big)=
n^2\big[H\big(\pfrac{x+1}{n}\big)+H\big(\pfrac{x-1}{n}\big)-2 H\big(\pfrac{x}{n}\big)\big]
\end{equation*}
be the discrete laplacian. Summing and subtracting suitable increments of $H$ in \eqref{eq3.5},
$ r_t^{n,\beta}$ can be rewritten as
\begin{equation*}\label{eq3.6}
\begin{split}
&\int_{0}^{t}\!\!\!\frac{1}{\sqrt{n}}\Big\{\! \alpha n^{2-\beta}\big[H(\pfrac{0}{n})-H(\pfrac{-1}{n})\big]
-n^2\big[H(\pfrac{-1}{n})-H(\pfrac{-2}{n})\big]-(\Delta_n H)(\pfrac{-1}{n})\!\Big\}\bar{\eta}_{sn^2}(-1)\,ds\\
\!+&\!\int_{0}^{t}\!\!\!\frac{1}{\sqrt{n}}\Big\{\!n^2\big[H(\pfrac{1}{n})-H(\pfrac{0}{n})\big]
-\alpha n^{2-\beta}\big[H(\pfrac{0}{n})-H(\pfrac{-1}{n})\big]-(\Delta_n
H)(\pfrac{0}{n})\!\Big\}\bar{\eta}_{sn^2}(0)\,ds,
\end{split}
\end{equation*}
plus a negligible term in $L^2(\bb P_\rho^\beta)$,
since $H$ is smooth and therefore $\Delta_n H$ is bounded. Then, we have that
\begin{equation*}
 r_t^{n,\beta}=\int_{0}^{t}\frac{1}{\sqrt{n}} (\alpha n^{2-\beta}-n^2)\big[H\big(\pfrac{0}{n}\big)-H\big(\pfrac{-1}{n}\big)\big]
\big(\bar{\eta}_{sn^2}(-1)-\bar{\eta}_{sn^2}(0)\big)\,ds.
\end{equation*}
Since $n\big[H\big(\pfrac{0}{n}\big)-H\big(\pfrac{-1}{n}\big)\big]$ is bounded, in order to show that  $r_t^{n,\beta}$ goes
to zero in $L^2(\mathbb{P}_\rho^\beta)$ as $n\rightarrow{+\infty}$, it is enough to show that
\begin{equation}\label{eq1}
\lim_{n\rightarrow{+\infty}}\mathbb{E}_\rho^\beta\Big[\Big(\int_{0}^{t}\sqrt{n}
\big\{\bar \eta_{sn^2}(-1)-\bar \eta_{sn^2}(0)\big\}\,ds\,\Big)^2\Big]=0.
\end{equation}
For that purpose we will make use of a comparison with empirical averages on boxes of a suitable size. Let
\begin{equation}\label{empirical average}
 \bar \eta^{\ell}(x)\;=\;\frac{1}{\ell}\sum_{y=x}^{x+\ell-1}\bar \eta(y),
\end{equation}
denote the centered empirical average of particles in a box of size $\ell$. Summing and subtracting the empirical mean at the sites $-1$
and $0$,
and applying the elementary inequality $(a+b+c)^2\leq 4(a^2+b^2+c^2)$, we bound the
 expectation in \eqref{eq1} from above by
\begin{equation*}\label{eq3.8}
\begin{split}
&4\,\mathbb{E}_\rho^\beta\Big[\Big(\int_{0}^{t}\sqrt{n}
\big\{\bar \eta_{sn^2}(-1)-\bar \eta_{sn^2}^\ell(-1)\big\}\,ds\,\Big)^2\,\Big]\\
+&4\,\mathbb{E}_\rho^\beta\Big[\Big(\int_{0}^{t}\sqrt{n}
\big\{\bar \eta_{sn^2}^{\ell}(-1)-\bar \eta_{sn^2}^\ell(0)\big\}\,ds\,\Big)^2\,\Big]\\
+&4\,\mathbb{E}_\rho^\beta\Big[\Big(\int_{0}^{t}\sqrt{n}
\big\{\bar \eta_{sn^2}^{\ell}(0)-\bar \eta_{sn^2}(0)\big\}\,ds\,\Big)^2\,\Big].
\end{split}
\end{equation*}
In order to estimate the first  expectation we use Lemma \ref{2orderBG}
 which guarantees that it is bounded from above by $C t (\alpha n^{\beta-1}+\ell/n)$, where $C$ is a constant. By Remark \ref{useful remark} the third expectation is bounded from above by $C' t \ell/n$, where $C'$ is a constant.
 On the other hand, a simple computation shows that the remaining expectation
  is bounded from above by $\tilde C t^2n/\ell^2$, where $\tilde C$ is a constant. Putting together the previous computations, we have that
 \begin{equation}\label{estimate}
\mathbb{E}_\rho^\beta\Big[\Big(\int_{0}^{t}\sqrt{n}
\big\{\bar \eta_{sn^2}(-1)-\bar \eta_{sn^2}(0)\big\}\,ds\,\Big)^2\Big]\leq{4C \Big(t\alpha n^{\beta-1}+\frac{t\ell}{n}\Big)+
4\tilde C\frac{t^2n}{\ell^2}+4C' \frac{t\ell}{n}}.
\end{equation}
Choose $\ell:=\varepsilon n$.
   Therefore, letting $n\to+\infty$ and then $\varepsilon\to 0$, the claim \eqref{eq1} follows.

\quad

\noindent $\bullet$ Case $\beta=1$:

\smallskip

In this case, by the definition of $\mc S_\beta(\bb R)$, we have that $\alpha \Big(H(0^+)-H(0^-)\Big)=H^{(1)}(0^+)=H^{(1)}(0^-)$.
Since  $H$ is continuous from the right at $x=0$, we have that
\begin{equation}\label{eq231}
 H\big(\pfrac{0}{n}\big)-H\big(\pfrac{-1}{n}\big)= \big[H(0^+)-H(0^-)\big] +O(1/n)\,,
\end{equation}
and
\begin{equation}\label{eq232}
 n\big[H\big(\pfrac{-1}{n}\big)-H\big(\pfrac{-2}{n}\big)\big]= H^{(1)}(0^-) +O(1/n)\,.
\end{equation}

We claim that the first integral in \eqref{eq3.5} is of order $O(t/\sqrt{n})$.  Since $\bar{\eta}_{sn^2}(-1)$ is bounded by one, the modulus of the first integral in \eqref{eq3.5} is bounded by
\begin{equation*}
\begin{split}
&\int_0^t\Big|\frac{1}{\sqrt{n}}\Big\{\alpha n^{2-\beta}[H(\pfrac{0}{n})-H(\pfrac{-1}{n})]-n^2[H(\pfrac{-1}{n})-H(\pfrac{-2}{n})]\Big\}\Big|\,ds\\
=\; & \frac{t}{\sqrt{n}}\Big|\alpha n^{2-\beta}[H(\pfrac{0}{n})-H(\pfrac{-1}{n})]-n^2[H(\pfrac{-1}{n})-H(\pfrac{-2}{n})]\Big|
\end{split}
\end{equation*}
Replacing \eqref{eq231} and \eqref{eq232} in the expression above, we get
\begin{equation*}
\frac{t}{\sqrt{n}}\Big|\alpha n^{2-\beta}[H(0^+)-H(0^-)+O(1/n)]-n[H^{(1)}(0^-)+O(1/n)]\Big|\,.
\end{equation*}
Since $\beta=1$ and in this case
\begin{equation*}
\alpha[H(0^+)-H(0^-)]\;=\;H^{(1)}(0^-)\,,
\end{equation*}
it implies that the first integral  in \eqref{eq3.5} is bounded by
\begin{equation*}
\frac{t}{\sqrt{n}}\,|n\,O(1/n)|\;=\;O(t/\sqrt{n})\,.
\end{equation*}
The same holds for the second integral in \eqref{eq3.5}.

Hence, when $\beta=1$,
 the expression $r_t^{n,\beta}$ is $O(t/n)$, which vanishes as $n\rightarrow{+\infty}$.

\quad

\noindent $\bullet$ Case $\beta\in{(1,+\infty]}$:

\smallskip

By definition of $\mc S_\beta(\bb R)$, since $H^{(1)}(0^+)=H^{(1)}(0^-)=0$, then we can rewrite
\begin{equation*}
 r_t^{n,\beta}=\int_{0}^{t}\alpha n^{3/2-\beta}
\big[H\big(\pfrac{0}{n}\big)-H\big(\pfrac{-1}{n}\big)\big]
\big(\bar{\eta}_{sn^2}(-1)-\bar{\eta}_{sn^2}(0)\big)\,ds+O(t/\sqrt{n}).
\end{equation*}
 Since for this range of the parameter $\beta$, $H$ is not smooth at the point $0$, in order to prove the claim it is enough to
show that:

\begin{equation*}\label{eq3.11}
\lim_{n\rightarrow{+\infty}}\mathbb{E}_\rho^\beta\Big[\Big(\int_{0}^{t}n^{3/2-\beta}
\big\{\bar \eta_{sn^2}(-1)-\bar \eta_{sn^2}(0)\big\}\,ds\,\Big)^2\Big]=0.
\end{equation*}

 By Lemma \ref{2orderBG} and by summing and subtracting $\eta_{sn^2}^\ell(-1)$ and $\eta_{sn^2}^\ell(0)$ as done above in the case $\beta\in{[0,1)}$, we can bound the previous expectation by $C
(t\alpha n^{1-\beta}+t\ell n^{1-2\beta}+t^2n^{3-2\beta}/\ell^2)$, where $C$ is a constant. Choose $\ell:=\varepsilon n$. Therefore, letting $n\to+\infty$ and then
$\varepsilon\to 0$, the claim follows.

\end{proof}

Now, recall from \eqref{martingale decomposition} that $\mc M^n_t(H)$
is a martingale.
 In the following section we prove that the sequence $\{\mc Y_t^n; t \in [0,T]\}_{n \in \bb N}$ is tight. Moreover, we prove that
the sequences $\{\mc I_t^n; t \in [0,T]\}_{n \in \bb N}$ and $\{\mc{M}_t^n; t \in [0,T]\}_{n \in \bb N}$ are tight. Assuming last
results, let $\{k_n\}_{n\in \bb N}$ be a subsequence such that all the sequences $\{\mc Y_t^{k_n}; t \in [0,T]\}_{n \in \bb N}$,
$\{\mc I_t^{k_n}; t \in [0,T]\}_{n \in \bb N}$ and $\{ \mc M_t^{k_n}; t \in [0,T]\}_{n \in \bb N}$ converge. Let $\{\mc Y_t; t \in
[0,T]\}$, $\{\mc I_t; t \in [0,T]\}$ and $\{\mc M_t; t \in [0,T]\}$ denote the limit of those sequences, respectively.

 We want to prove that
$\{\mc Y_t; t \in [0,T]\}$ is in $\mc C([0,T],\mathcal{S}'_{\beta}(\bb R))$ and also that for $H\in{\mc
S_\beta(\bb R)}$:
\begin{equation*}
\mc M_t(H)= \mc Y_t(H) -\mc Y_0(H) -  \int_0^t \mc Y_s(\A H)\,ds
\end{equation*}
is a martingale with quadratic variation given by $2\chi(\rho)t\|\B H\|_{2,\beta}^2$.
Fix $H\in{\mc S_\beta(\bb R)}$. Since we have that for each $n\in{\bb N}$, $\mc M_t^{k_n}(H)$
is a martingale, we want to show that passing to the limit in $n$ we obtain that $\mc M_t(H)$ is a martingale. We notice that the limit in distribution of a uniformly integrable sequence of martingales is a martingale, see Proposition 4.6 of \cite{gj}. Therefore, it is enough to show that
$\{\mc M_t^{k_n}(H)\}_{n \in \bb N}$ is uniformly integrable. To this end we notice that by Lemma \ref{martingaleL2bounds} we have
that
\begin{equation}\label{eq234}
\lim_{n\to{+\infty}}\mathbb{E}_\rho^\beta[(\mc M_t^{k_n}(H))^2]=
2\chi(\rho)\,t \|\B H\|_{2,\beta}^2,
\end{equation}
which is enough to assure the uniform integrability.

We claim that  the quadratic variation of the martingale $\{\mc
M_t(H)\,;\,0\leq t\leq T\}$ is given by $\{2\chi(\rho)t\|\B
H\|_{2,\beta}^2\,,\,0\leq t\leq T\}$. To this end,  observe that
\begin{equation}\label{eq251}
\{(\mc M_t^{k_n}(H))^2-\<\mc M^{k_n}(H)\>_t\,;\,0\leq t\leq T\}
\end{equation}
 is a martingale for each $n\in{\bb N}$. Applying \eqref{eq234}, the
limit in distribution of \eqref{eq251} as $n\to+\infty$ is
 \begin{equation*}
\{ (\mc M_t(H))^2-2\chi(\rho)t\|\B H\|_{2,\beta}^2\,;\,0\leq t\leq T\}.
\end{equation*}
Thus, it suffices to show that expression above is a martingale.
Again, let us use the fact that the limit in distribution of an
uniformly integrable sequence of martingales is a martingale, now for
the sequence \eqref{eq251}.

By \eqref{eq234},  the sequence $\<\mc M^{k_n}(H)\>_t$ is uniformly
integrable. Hence,  we only have to prove that $(\mc{M}_t^{k_n}(H))^2$
is uniformly integrable. For that purpose we
prove that $\mathbb{E}_\rho^\beta[(\mc M_t^{k_n}(H))^4]$ is bounded by a
constant that does not depend on $n$.
Now, we can employ, for example, Lemma 3 of \cite{dg} which says that
there exists a constant $C$ such that
\begin{equation*}
\mathbb{E}_\rho^\beta[(\mc
M_t^{k_n}(H))^4]\leq{C\Big(\mathbb{E}_\rho^\beta[(\mc
M_t^{k_n}(H))^2]+\mathbb{E}_\rho^\beta\Big[
\sup_{0\le t\le T}
\Big| \mc M_t^{k_n}(H) - \mc M_{t^-}^{k_n}(H)\Big|^4 \Big]\Big)}.
\end{equation*}
By Lemma \ref{martingaleL2bounds} the first term on the left hand side
of the previous inequality is bounded. On the other hand,
since
$$\sup_{0\le t\le T}| \mc M_{t}^{k_n} (H) - \mc M_{t^-}^{k_n} (H)|=\sup_{0\le
t\le T}| \mc Y_{t}^{k_n} (H) - \mathcal
Y_{t^-}^{k_n}(H)|\leq{\frac{C(H)}{\sqrt {k_n}}}\,,$$ the second term on the
right hand side of the previous inequality is also bounded, this
finishes the proof.
\subsubsection{Convergence at initial time}

\begin{proposition} \label{convergence at time zero}

 $\mc Y^n_0$ converges in distribution to $\mc Y_0$,
where $\mc Y_0$
 is a gaussian field with mean zero and covariance given by \eqref{eq:covar1}.
\end{proposition}
\begin{proof}
We first claim that, for every $H\in{\mathcal{S}_\beta(\mathbb R})$ and every $t>{0}$,
\begin{equation*}\label{eq23}
\lim_{n\rightarrow{{+\infty}}}\log\mathbb{E}_\rho^\beta\Big[\exp\{i\theta{\mathcal{Y}^n_{0}(H)}\}\Big]=-\frac{\theta^2}{2}
\chi(\rho)\int_{\mathbb{R}} H^2(u)du\,.
\end{equation*}
Since $\nu_\rho$ is a Bernoulli product measure,
\begin{equation*}
\begin{split}
\log\mathbb{E}_\rho^\beta[\exp\{i\theta\mathcal{Y}^n_{0}(H)\}]&=
\log\mathbb{E}_\rho^\beta\Big[\exp\Big\{\frac{i\theta}{\sqrt{n}}\sum
_{x	\in{\bb Z}}\; \bar{\eta}_0(x)H\Big(\frac{x}{n}\Big)\Big\}\Big]\\
&=\sum_{x\in{\bb Z}}\log\mathbb{E}_\rho^\beta\Big[\exp\Big\{\frac{i\theta}{\sqrt{n}}\;
\bar{\eta}_0(x)H\Big(\frac{x}{n}\Big)\Big\}\Big]\,.
\end{split}
\end{equation*}
Since $H$ is smooth except possibly at $x=0$, using Taylor's expansion the right hand side of last expression is equal to
\begin{equation*}
-\frac{\theta^2}{2n}\sum_{x\in{\bb Z }}H^{2}\Big(\frac{x}{n}\Big)\chi(\rho)+O(1/\sqrt n).
\end{equation*}
Taking the limit as $n\rightarrow{+\infty}$ and using the continuity of $H$, the proof of the claim ends. Replacing $H$ by a
linear combination of functions and recalling the Cr\'amer-Wold device, the proof finishes.
\end{proof}

\begin{remark}
We notice that the result stated above holds true for $\mathcal{Y}_t$ for any $t\in [0,T]$. In particular we conclude that the
gaussian white noise is a stationary solution of \eqref{OU}, for any $\beta\in[0,{+\infty}]$.
\end{remark}
\subsection{Tightness}

Here we prove tightness of the process $\{\mc Y_t^n; t \in [0,T]\}_{n \in \bb N}$. At first we notice that by the Mitoma's
criterion and Proposition \ref{frechet}, it is enough to prove tightness of the sequence of real-valued processes $\{\mc Y_t^n(H); t \in [0,T]\}_{n \in \bb N}$,
for $H\in{\mc {S}_\beta(\bb R)}$.

\begin{proposition}[Mitoma's criterion \cite{Mit}]
\quad

 A sequence $\{x_t;t \in [0,T]\}_{n \in \bb N}$ of processes in $\mc D([0,T],\mc {S}'_\beta(\bb R))$ is tight with respect to the
Skorohod topology if and only if the sequence $\{x_t(H);t \in [0,T]\}_{n \in \bb N}$ of real-valued processes is tight with
respect to the Skorohod topology of $\mc D([0,T], \bb R)$, for any $H \in \mc {S}_\beta(\bb R)$.
\end{proposition}
Now, to show tightness of the real-valued process we use the following Aldous' criterion.
\begin{proposition}
 A sequence $\{x_t; t\in [0,T]\}_{n \in \bb N}$ of real-valued processes is tight with respect to the Skorohod topology of $\mc
D([0,T],\bb R)$ if:
\begin{itemize}
\item[i)]
$\displaystyle\lim_{A\rightarrow{+\infty}}\;\limsup_{n\rightarrow{+\infty}}\;\mathbb{P}_\rho^\beta\Big(\sup_{0\leq{t}\leq{T}}|x_{t
} |>A\Big)\;=\;0\,,$

\item[ii)] for any $\varepsilon >0\,,$
 $\displaystyle\lim_{\delta \to 0} \;\limsup_{n \to {+\infty}} \;\sup_{\lambda \leq \delta} \;\sup_{\tau \in \mc T_T}\;
\mathbb{P}_\rho^\beta(|
x_{\tau+\lambda}- x_{\tau}| >\varepsilon)\; =\;0\,,$
\end{itemize}
where $\mc T_T$ is the set of stopping times bounded by $T$.
\end{proposition}

Fix $H\in{\mc S_\beta(\bb R)}$. By \eqref{martingale decomposition}, it is enough to prove tightness of $\{\mc Y_0^n(H)\}_{n \in
\bb N}$, $\{ \mc I_t^n(H); t \in [0,T]\}_{n \in \bb N}$, and $\{\mc M_t^n(H); t \in [0,T]\}_{n \in \bb N}$. By Proposition \ref
{convergence at time zero} the sequence of initial fields $\{\mc Y_0^n(H)\}_{n \in
\bb N}$ is obviously tight.
For the martingale term, the first claim of the Aldous' criterion is straightforwardly verified as an application of Doob's inequality
together with \eqref{var_quad}. By Lemma \ref{integralL2bounds}, the first claim can be easily checked for the integral term. It
remains to check the second claim, which is more demanding. For that purpose,
fix a stopping time $\tau \in \mc T_T$. By Chebychev's inequality together with Lemma \ref{martingaleL2bounds} we have
that
\begin{equation*}
\begin{split}
\mathbb{P}_\rho^\beta\big(\big| \mc M_{\tau+\lambda}^n(H) - \mc M_\tau^n(H)\big| >\varepsilon\big)
	&\leq \frac{1}{\varepsilon^2} \mathbb{E}_\rho^\beta\big[ \big( \mc M_{\tau+\lambda}^n(H) - \mc M_\tau^n(H)\big)^2\big]\\
	&\leq \frac{1}{\varepsilon^2} 2\chi(\rho)\,\lambda \Vert \B H \Vert^2_{2,\beta}\\
	&\leq \frac{1}{\varepsilon^2} 2\chi(\rho)\,\delta \Vert  \B H\Vert^2_{2,\beta},
\end{split}
\end{equation*}
which vanishes as $\delta\rightarrow{0}$.
 In order to check the second claim for the integral term, we use the same argument as above together with Lemma
\ref{integralL2bounds} to have that
 \begin{equation*}
\begin{split}
\mathbb{P}_\rho^\beta\big(\big|  \mc I_{\tau+\lambda}^n(H) -  \mc I_\tau^n(H)\big| >\varepsilon\big)
	&\leq \frac{1}{\varepsilon^2} \mathbb{E}_\rho^\beta\big[ \big( \mc I_{\tau+\lambda}^n(H) - \mc I_\tau^n(H)\big)^2\big]\\
	&\leq \frac{80t}{\varepsilon^2}\delta\,\chi(\rho)\Vert \B H\Vert^2_{2,\beta},
\end{split}
\end{equation*}
which vanishes as $\delta\rightarrow{0}$.
This finishes the proof of tightness.

\section{Semigroup results}\label{s4}

Here we present the deduction of the explicit formula for the semigroup $T_t^\alpha$ associated to the following heat equation with a boundary condition of Robin's type
\begin{equation}\label{pdeApendice}
\left\{
\begin{array}{ll}
 \partial_t u(t,x) = \; \partial^2_{xx} u(t,x), &t \geq 0,\, x \in \mathbb R\backslash\{0\}\\
\p_x u(t,0^+)=\p_x u(t,0^-)= \alpha \{u(t,0^+)-u(t,0^-)\},  &t \geq 0\\
 u(0,x) = \; g(x), &x \in \mathbb R.
\end{array}
\right.
\end{equation}
Let $T_t$ be the semigroup associated to the heat equation \eqref{pde1}.
Let $\tilde{T}_t^\alpha$ be the semigroup related to the following partial differential equation on the half-line:
\begin{equation}\label{pde4}
\left\{
\begin{array}{ll}
\partial_t u(t,x) = \; \partial^2_{xx} u(t,x), &t \geq 0,\, x >0\\
\p_x u(t,0^+)= 2\alpha u(t,0^+),  &t \geq 0\\
u(0,x) = \; g(x), &x >0.
\end{array}
\right.
\end{equation}
A direct verification shows that
\begin{equation}\label{eq34}
 T_t^\alpha g(x)\;=\;
 \begin{cases}
  T_tg_{\textrm{even}}(x)+\tilde{T}_t^\alpha g_{\textrm{odd}}(x)\,,&\textrm{for } x>0\,,\\
T_tg_{\textrm{even}}(x)-\tilde{T}_t^\alpha g_{\textrm{odd}}(-x)\,,&\textrm{for } x<0\,,\\
  \end{cases}
\end{equation}
is solution of \eqref{pdeApendice}. Since the semigroup $T_t$ has the classical expression given in \eqref{sem heat eq}, we are therefore left to deduce an explicit  expression for $\tilde{T}_t^\alpha$. Denote by $u$ the solution of \eqref{pde4} and consider
$v=2\alpha u-\p_x u\,,$
which   is the solution of the following equation
\begin{equation*}
\left\{
\begin{array}{ll}
\partial_t v(t,x) = \; \partial^2_{xx} v(t,x), &t \geq 0,\, x >0\\
v(t,0^+)= 0,  &t \geq 0\\
v(0,x) = \; v_0(x), &x >0.
\end{array}
\right.
\end{equation*}
with $v_0(x)=2\alpha g(x)-\p_x g(x)$. Last equation is the heat
equation with a boundary condition of Dirichlet's type.  The semigroup $T^{\textrm{Dir}}_t v_0(x)$ associated to last equation, is classical and  is
given by
\begin{equation}\label{sem dir}
 T^{\textrm{Dir}}_tv_0(x)\;=\;\frac{1}{\sqrt{4\pi t}}\int_0^{{+\infty}}\Big[e^{-\frac{(x-y)^2}{4t}}-e^{-\frac{(x+y)^2}{4t}}\Big]v_0(y)\,dy\,.
\end{equation}

Then, we get to
\begin{equation*}
v(t,x)=\frac{1}{\sqrt{4\pi t}}\int_{0}^{+\infty} \Big[e^{-\frac{(x-y)^2}{4t}}-e^{-\frac{(x+y)^2}{4t}}\Big]\,\{2\alpha g(x)-\p_x
g(x)\}\,dy\,.
\end{equation*}
Solving the ordinary linear differential equation $v=2\alpha u-\p_x u$, we get
\begin{equation*}
 u(t,x)=e^{2\alpha x}\int_{x}^{+\infty} e^{-2\alpha z} \,v(t,z)\,dz\,.
\end{equation*}
From the last two formulas, we arrive at
\begin{equation*}
 \tilde{T}_t^\alpha g(x)=\frac{e^{2\alpha x}}{\sqrt{4\pi t}}\int_x^{+\infty} e^{-2\alpha z}\int_0^{+\infty}
\Big[e^{-\frac{(z-x)^2}{4t}}- e^{-\frac{(z+x)^2}{4t}}\Big] \Big(2\alpha g(y)-\p_y g(y)\Big)\,dy\,dz\,.
\end{equation*}
Finally, an  integration by parts on the term of the integral above involving $\p_y g$ yields
\begin{equation}\label{semi_alpha}
 \tilde{T}_t^\alpha g(x)=\frac{e^{2\alpha x}}{\sqrt{4\pi t}}\int_x^{+\infty} e^{-2\alpha z}\int_0^{+\infty}
\Big[(\pfrac{z-y+4\alpha t}{2t})e^{-\frac{(z-y)^2}{4t}}+(\pfrac{z+y-4\alpha t}{2t})e^{-\frac{(z+y)^2}{4t}}\Big]
g(y)\,dy\,dz\,.
\end{equation}
Putting this formula together with \eqref{eq34} and \eqref{sem heat eq} we get the statement of
Proposition \ref{prop23}.

In possess of the expression of all the semigroups, we can proceed to the
\begin{proof}[Proof of Proposition \ref{conver_pde}]
Recall \eqref{eq34}. We claim that
\begin{equation}\label{Neu}
 \lim_{\alpha\to 0} \tilde{T}^\alpha_tg_{{\textrm{odd}}}(x)=T^{\textrm{Neu}}_tg_{\textrm{odd}}(x)
\end{equation}
and
\begin{equation}\label{Dir}
 \lim_{\alpha\to {+\infty}} \tilde{T}^\alpha_tg_{{\textrm{odd}}}(x)=T^{\textrm{Dir}}_tg_{\textrm{odd}}(x)\,,
\end{equation}
where $T^{\textrm{Dir}}_t$ is given by \eqref{sem dir}.
We observe that proving \eqref{Neu} and \eqref{Dir} is enough to conclude the proof, since it is of easy verification that
\begin{equation*}
 T_tg(x)\;=\;
 \begin{cases}
  T_tg_{\textrm{even}}(x)+T_t^{\textrm{Dir}} g_{\textrm{odd}}(x)\,,&\textrm{for } x>0\,,\\
T_tg_{\textrm{even}}(x)-T_t^{\textrm{Dir}} g_{\textrm{odd}}(-x)\,,&\textrm{for } x<0\,,\\
  \end{cases}
\end{equation*}
and
\begin{equation*}
 T_t^{\textrm{Neu}}g(x)\;=\;
 \begin{cases}
  T_tg_{\textrm{even}}(x)+T_t^{\textrm{Neu}} g_{\textrm{odd}}(x)\,,&\textrm{for } x>0\,,\\
T_tg_{\textrm{even}}(x)-T_t^{\textrm{Neu}} g_{\textrm{odd}}(-x)\,,&\textrm{for } x<0\,.\\
  \end{cases}
\end{equation*}
 Since $g_{\textrm{odd}}$ will have no special role in the convergences \eqref{Neu} and \eqref{Dir}, we will write just
$g$ instead.  We start by showing \eqref{Neu}. First, we rewrite \eqref{semi_alpha} to get to
\begin{equation*}
\begin{split}
\tilde{T}_t^\alpha g(x)=&\frac{e^{2\alpha x}}{\sqrt{4\pi t}}\int_x^{+\infty} e^{-2\alpha z}\int_0^{+\infty}
\Big[(\pfrac{z-y}{2t})e^{-\frac{(z-y)^2}{4t}}+(\pfrac{z+y}{2t})e^{-\frac{(z+y)^2}{4t}}\Big]
g(y)\,dy\,dz\\
+&\frac{2\alpha e^{2\alpha x}}{\sqrt{4\pi t}}\int_x^{+\infty} e^{-2\alpha z}\int_0^{+\infty}
\Big[e^{-\frac{(z-y)^2}{4t}}-e^{-\frac{(z+y)^2}{4t}}\Big]
g(y)\,dy\,dz \,.
\end{split}
\end{equation*}
When $\alpha\to 0$, the second parcel on the right hand  side of previous equation vanishes. Thus, we are concerned only with
the first parcel. Its limit when $\alpha\to 0$ is
\begin{equation*}
 \frac{1}{\sqrt{4\pi t}}\int_0^{+\infty} \int_x^{+\infty}
\Big[(\pfrac{z-y}{2t})e^{-\frac{(z-y)^2}{4t}}+(\pfrac{z+y}{2t})e^{-\frac{(z+y)^2}{4t}}\Big]
g(y)\,dy\,dz\,.
\end{equation*}
Applying Fubini's Theorem to last expression above gives
\begin{equation*}
\frac{1}{\sqrt{4\pi t}} \int_0^{+\infty} g(y)\int_x^{+\infty}
\Big[(\pfrac{z-y}{2t})e^{-\frac{(z-y)^2}{4t}}+(\pfrac{z+y}{2t})e^{-\frac{(z+y)^2}{4t}}\Big]
\,dz\,dy\,.
\end{equation*}
Solving the integral in $z$, we get that last expression  equals to $T^{\textrm{Neu}}_tg(x)$, as claimed.\medskip

Now we prove \eqref{Dir}.  We begin by splitting \eqref{semi_alpha} as
\begin{equation}\label{eq38}
\begin{split}
\tilde{T}_t^\alpha g(x)=&2\alpha e^{2\alpha x}\int_x^{+\infty} e^{-2\alpha z}\frac{1}{2\alpha}\int_0^{+\infty}
\frac{1}{\sqrt{4\pi
t}}\Big[(\pfrac{z-y}{2t})e^{-\frac{(z-y)^2}{4t}}+(\pfrac{z+y}{2t})e^{-\frac{(z+y)^2}{4t}}\Big]
g(y)\,dy\,dz\\
 +&2\alpha e^{2\alpha x}\int_x^{+\infty} e^{-2\alpha z}\int_0^{+\infty} \frac{1}{\sqrt{4\pi t}}
\Big[e^{-\frac{(z-y)^2}{4t}}-e^{-\frac{(z+y)^2}{4t}}\Big]
g(y)\,dy\,dz \,.
\end{split}
\end{equation}
Since
\begin{equation*}
 \int_x^{+\infty} e^{-2\alpha z}\,dz \;=\;\frac{e^{-2\alpha x}}{2\alpha }\,,
\end{equation*}
we can see that the first parcel on right hand side of \eqref{eq38} is an average of the function
\begin{equation}\label{eq39}
\frac{1}{2\alpha}\int_0^{+\infty}
\frac{1}{\sqrt{4\pi
t}}\Big[(\pfrac{z-y}{2t})e^{-\frac{(z-y)^2}{4t}}+(\pfrac{z+y}{2t})e^{-\frac{(z+y)^2}{4t}}\Big]
g(y)\,dy
\end{equation}
over the finite measure ${\bf{1}}_{[x,{+\infty})}(z)e^{-2\alpha z}\,dz$. Since \eqref{eq39} goes to zero when
$\alpha\to{+\infty}$, we are only concerned with the second parcel in \eqref{eq38}. By Fubini's Theorem, it is equal to
\begin{equation*}
  \frac{ e^{2\alpha x}}{\sqrt{4\pi t}}\int_0^{+\infty} g(y)\int_x^{+\infty} 2\alpha e^{-2\alpha z}
\Big[e^{-\frac{(z-y)^2}{4t}}-e^{-\frac{(z+y)^2}{4t}}\Big]\,dz\,dy \,.
\end{equation*}
Performing an integration by parts to the integral in $z$ yields
\begin{equation*}
\begin{split}
 &\frac{ e^{2\alpha x}}{\sqrt{4\pi t}}\int_0^\infty g(y)\Bigg[ -e^{-2\alpha z}
\Big(e^{-\frac{(z-y)^2}{4t}}-e^{-\frac{(z+y)^2}{4t}}\Big)\Big\vert_{z=x}^{z={+\infty}}\\
 &+\int_x^{+\infty} e^{-2\alpha
z}\Big(-\pfrac{(y-z)}{2t}e^{-\frac{(z-y)^2}{4t}}-\pfrac{(y+z)}{2t}e^{-\frac{(z+y)^2}{4t}}\Big)\,dz\Bigg]dy\,, \\
\end{split}
\end{equation*}
which is equal to
\begin{equation*}
 T_t^{\textrm{Dir}}g(x)-
 \frac{ e^{2\alpha x}}{\sqrt{4\pi t}}\int_0^{+\infty} g(y)\Bigg[ \int_x^{+\infty} e^{-2\alpha
z}\Big(\pfrac{(y-z)}{2t}e^{-\frac{(z-y)^2}{4t}}+\pfrac{(y+z)}{2t}e^{-\frac{(z+y)^2}{4t}}\Big)\,dz\Bigg]dy \,.
\end{equation*}
Multiplying and dividing the integral term above by $2\alpha$, and then applying the same argument on the average previously used,
we get that the limit when $\alpha\to+\infty$  is given by $T_t^{\textrm{Dir}}g(x)$, finishing the proof of the pointwise
convergence.

In order to conclude the $L^p(\bb R)$ convergence, we notice that the semigroups are written in terms of the gaussian kernel, from which  is not difficult
to get a uniform bound in $\alpha$. Invoking the Dominated Convergence Theorem the proof finishes.
\end{proof}

\section{Proof of Proposition \ref{pp1}}\label{s5}
The existence of the Ornstein-Uhlenbeck process solution of \eqref{eq Ou} was already proved in  Section \ref{s3}. In this section we guarantee that there exists at most one random element
$\Y$
taking values in $\mc C([0,T],\mathcal{S}'_{\beta}(\bb R))$ such that  \emph{i)} and \emph{ii)} of Proposition
\ref{pp1} hold. The next lines follow closely from  \cite[page 307]{kl}. The key result is the equality $T^\beta_{t+\eps} H-T^\beta_t H=\eps \Delta_\beta T_t^\beta H+o(\eps)$,
which is well-known  for $\beta\in{[0,1)}$. Since the semigroups $T^\alpha_t$ and $T_t^{\textrm{Neu}}$
are written in terms of the gaussian kernel, the same property holds for them, provided $H$ is in the corresponding domain.
In what follows, the same arguments apply for all cases of $\beta$, and we
will just write $T^\beta_t$ for the corresponding semigroup.\medskip

Fix $H\in\mc S_\beta(\bb R)$ and $s>0$. Recall from \eqref{Bmotion} that  $ \mc M_t(H)(2\chi(\rho)\|\B H\|_{2,\beta}^2)^{-1/2}$
is a standard Brownian motion. Therefore, by It\^o's Formula, the process $\{X_t^s(H)\,;\,t\geq s\}$ defined by
\begin{equation*}
 X_t^s(H)=\exp\Bigg\{\frac{1}{2}(t-s)\Vert \B H\Vert^2_{2,\beta}+ i\Big( \Y_t(H) - \Y_s(H) -\int_s^t\Y_r(\A
H)\,dr\Big)\Bigg\}
\end{equation*}
is a martingale. Fix $S>0$. We affirm now that the process $\{Z_t\,;\, 0\leq t\leq S\}$  defined by
\begin{equation*}
 Z_t=\exp\Big\{ \frac{1}{2}\int_0^t\Vert \B T^\beta_{S-r} H\Vert^2_{2,\beta}\,dr +i\,\Y_t(T^\beta_{S-t} H)\Big\}
\end{equation*}
is also a martingale. To prove this, consider two times $0\leq t_1<t_2\leq S$ and a partition of the interval $[t_1,t_2]$ in $n$
intervals of equal size, or else, $t_1=s_0<s_1<\cdots<s_n=t_2\,,$
with $s_{j+1}- s_j=(t_2-t_1)/n$. Observe now that
\begin{equation*}
\begin{split}
\prod_{j=0}^{n-1} X_{s_{j+1}}^{s_j}(T^\beta_{S-s_j}H)=&\exp\Bigg\{  \frac{1}{2n}\sum_{j=0}^{n-1}
\Vert \B T^\beta_{S-s_j}H\Vert^2_{2,\beta}\\
& +i\,\sum_{j=0}^{n-1}
\Big( \Y_{s_{j+1}}(T^\beta_{S-s_j}H) - \Y_{s_j}(T^\beta_{S-s_j}H) -\int_{s_j}^{s_{j+1}}\Y_r(\A
T^\beta_{S-s_j}H)\,dr\Big)\Bigg\}\,.\\
\end{split}
\end{equation*}
As $n\to +\infty$, the first sum inside the exponential above converges to
\begin{equation*}
\frac{1}{2}\int_{t_1}^{t_2}\Vert \B T^\beta_{S-r}H\Vert^2_{2,\beta} \,dr\,,
\end{equation*}
because it is a Riemann sum. The second sum inside the exponential can be rewritten as
\begin{equation*}
 \Y_{t_2}(T^\beta_{S-t_2+\frac{1}{n}}H)-\Y_{t_1}(T^\beta_{S-t_1}H)+ \sum_{j=1}^{n-1}\!\!
\Big( \Y_{s_{j}}(T^\beta_{S-s_{j-1}}\!\!\!\!H-T^\beta_{S-s_j}H) -\int_{s_j}^{s_{j+1}}\!\!\Y_r(\A
T^\beta_{S-s_j}H)\,dr\Big)\,.
\end{equation*}
Since $\Y\in \mc C([0,T],\mc S_\beta'(\bb R))$, since $T^\beta_t H$ is  continuous in time and applying the expansion
$T_{t+\eps}^\beta H-T^\beta_t H=\eps \A T_t^\beta H+o(\eps)$,
we conclude that the almost sure limit of the previous expression is just
$\Y_{t_2}(T^\beta_{S-t_2}H)-\Y_{t_1}(T^\beta_{S-t_1}H)\,.$
Thus, we have obtained that
\begin{equation*}
 \lim_{n\to {+\infty}} \prod_{j=0}^{n-1} X_{s_{j+1}}^{s_j}(T^\beta_{S-s_j}H)=\exp\Bigg\{
 \frac{1}{2}\int_{t_1}^{t_2}\Vert \B T^\beta_{S-r}H\Vert^2_{2,\beta} \,dr+
 i\Big(\Y_{t_2}(T^\beta_{S-t_2}H)-\Y_{t_1}(T^\beta_{S-t_1}H)\Big)\Bigg\},
\end{equation*}
 which equals to $\frac{ Z_{t_2}}{Z_{t_1}}$ almost surely. Since the complex exponential is bounded, the Dominated Convergence
Theorem ensures also the $L^1$
convergence, which on the other hand implies that
\begin{equation*}
 \bb E_\rho^\beta\Big[G\,\frac{Z_{t_2}}{Z_{t_1}} \Big] =\lim_{n\to {+\infty}} \bb E^\beta_\rho\Big[G\,\prod_{j=0}^{n-1}
X_{s_{j+1}}^{s_j}(T^\beta_{S-s_j}H) \Big]\,,
\end{equation*}
for any bounded function $G$. Take $G$ bounded and  $\mc F_{t_1}$-measurable. Since for any $H\in \mc
S_\beta(\bb R)$, the process $X_t^s(H)$ is a martingale, by taking the
conditional expectation with respect to
$\mc F_{s_{n-1}}$ we can see that
\begin{equation*}
 \bb E_\rho^\beta\Big[G\,\prod_{j=0}^{n-1}
X_{s_{j+1}}^{s_j}(T^\beta_{S-s_j}H) \Big]=\bb E_\rho^\beta\Big[G\,\prod_{j=0}^{n-2}
X_{s_{j+1}}^{s_j}(T^\beta_{S-s_j}H) \Big]\,.
\end{equation*}
By induction, we conclude that
\begin{equation*}
  \bb E_\rho^\beta\Big[G\,\frac{Z_{t_2}}{Z_{t_1}} \Big] = \bb E_\rho^\beta\Big[G \Big]\,,
\end{equation*}
for any $G$ bounded and $\mc F_{t_1}$-measurable, which proves that $\{Z_t\,;\,t\geq 0\}$ is a martingale. From
$\bb E_\rho^\beta[Z_{t} | \mc F_s] = Z_{s}$, we get
\begin{equation*}
\begin{split}
 &\bb E_\rho^\beta\Big[\exp\Big\{ \frac{1}{2}\int_0^t\Vert \B T^\beta_{S-r} H\Vert^2_{2,\beta}\,dr
+i\,\Y_t(T^\beta_{S-t} H)\Big\}\Big\vert \mc F_s\Big]\\
&= \exp\Big\{ \frac{1}{2}\int_0^s\Vert \B T^\beta_{S-r} H\Vert^2_{2,\beta}\,dr +i\,\Y_s(T^\beta_{S-s} H)\Big\}\,,
\end{split}
\end{equation*}
which in turn gives
\begin{equation*}
\bb E_\rho^\beta\Big[\exp\Big\{ i\,\Y_t(T^\beta_{S-t}
H)\Big\}\Big\vert \mc F_s\Big] = \exp\Big\{ -\frac{1}{2}\int_s^t\Vert \B T^\beta_{S-r} H\Vert^2_{2,\beta}\,dr
+i\,\Y_s(T^\beta_{S-s} H)\Big\}\,.
\end{equation*}
Since $T_{S-s}^\beta H=T_{t-s}^\beta T_{S-t}^\beta H$, performing  a change of variables in  $H$ and then a change of variables in time, we are lead to
\begin{equation*}
\bb E_\rho^\beta\Big[\exp\Big\{ i\,\Y_t(H)\Big\}\Big\vert \mc F_s\Big] = \exp\Big\{ -\frac{1}{2}\int_0^{t-s}\Vert \B
T^\beta_{r} H\Vert^2_{2,\beta}\,dr
+i\,\Y_s(T^\beta_{t-s} H)\Big\}\,.
\end{equation*}
Replacing $H$ by $xH$, where $x\in \bb R$, we get that conditionally to $\mc F_s$, the random variable $\Y_t(H)$ has  gaussian
distribution of mean $\Y_s(T^\beta_{t-s}H)$ and variance $\int_0^{t-s}\Vert \B T^\beta_{r} H\Vert^2_{2,\beta}\,dr$.
Successive conditioning  implies the uniqueness of the finite dimensional distributions of  the process $\{\Y_t(H)\,;\,t\in{[0,T]}\}$,
which in turn gives uniqueness in law of the random element $\Y$, finishing the proof.

\section{Central Limit Theorem for the Current}\label{s6}
In this section we follow \cite{g,jl,rv}. Recall the definition of the current $J^n_{x,x+1}(t)$ given in Subsection
\ref{sub_eq}. Since the system starts from the equilibrium $\nu_\rho$ and the dynamics is symmetric, then $\mathbb{E}_\rho^\beta[{J}^n_{x,x+1}(t)]=0$,  for any
time $t\geq{0}$ and any site $x\in{\mathbb{Z}}$.

For any $x\in{\mathbb{Z}}$, if the number of
particles in the configuration $\eta$ is finite,  we can  write
\begin{equation*}
{J}^n_{x,x+1}(t)=\sum_{y\geq{x+1}}\Big(\eta_{tn^2}(y)-\eta_{0}(y)\Big).
\end{equation*}
In such case, the current through the bond $\{\lfloor u n\rfloor-1,\lfloor u n\rfloor\}$ can be written in
terms of the density fluctuation field
$\mathcal{Y}_{t}^{n}$ as
\begin{equation*}
\frac{{J}^n_{u}(t)}{\sqrt{n}}=\mathcal{Y}_{t}^{n}(H_{u
})
-\mathcal{Y}_{0}^{n}(H_{u}),
\end{equation*}
where $H_{u}$ is the Heaviside function, or else, $H_{u}(x)={\bf 1}_{[u,{+\infty})}(x)$. Our goal is to take the limit as $n\to+\infty$ in the previous equality. At this point we face two problems. Firstly, the equality itself makes no sense unless the configuration $\eta$ has a finite numbers of particles. Secondly, the
Heaviside function does not belong to the space $\mc S_{\beta}(\bb R)$.
To overcome these difficulties, we  notice that by the conservation on the number of particles it holds that
\begin{equation}\label{J}
 J^n_{x-1,x}(t)-J^n_{x,x+1}(t)\;=\;\eta_t(x)-\eta_0(x).
\end{equation}
 Next, we define a sequence of functions $\{G^u_j\}_{j\in{\bb N}}$ such that
 $G^u_{j}(x):=(1-\frac{x-u}{j})^{+}H_u(x)$,
approximating the Heaviside function $H_u$. For these functions, the process $\mathcal{Y}_{t}^{n}(G^u_j)$ makes sense,
no matter the finiteness of the total number of particles.  A discrete
integration by parts together with \eqref{J} gives
\begin{equation*}
 \mathcal{Y}_{t}^{n}(G_j^u)-\mathcal{Y}_{0}^{n}(G_j^u)=\frac{1}{\sqrt{n}}
 \sum_{x\in \bb Z }\Big(G_j^u(\pfrac{x+1}{n})-G_j^u(\pfrac{x}{n})\Big)\, J^n_{x,x+1}(t)\,.
\end{equation*}
As $j\to{+\infty}$, the derivative of $G_j^u$ becomes zero except at the discontinuity point $x=u$. This motivates the next
lemma.
\begin{lemma}\label{prop for current}
For every $t\geq{0}$ and for every $\beta\in[0,{+\infty}]$,
\begin{equation*}
\lim_{j\rightarrow{+\infty}}\mathbb{E}_\rho^\beta\Big[\Big(\frac{{J}^n_{u}(t
)}{\sqrt{n}}-(\mathcal{Y}_{t}^{n}(G^u_{j})
-\mathcal{Y}_{0}^{n}(G^u_{j}))\Big)^2\Big]=0\,,
\end{equation*}
uniformly over $n$.
\end{lemma}
\begin{proof}
Recall \eqref{martingale decomposition} and \eqref{I}.
A simple computation together with \eqref{J} shows that
\begin{equation*}
\frac{J^n_{u}(t)}{\sqrt{n}}-(\mathcal{Y}_{t}^{n}(G^u_{j})
-\mathcal{Y}_{0}^{n}(G^u_{j}))=\mathcal{M}_{t}^{n}(H_u-G^u_{j})+\mc I_{t}^{n}(H_u-G^u_{j}).
\end{equation*}
 By the inequality $(x+y)^2\leq{2x^2+2y^2}$, in order to prove the lemma, it is enough to show that the second moment of the
two terms on the right hand side of the previous equality vanishes as $j\rightarrow{+\infty}$, uniformly over $n$.

Taking
$f=H_u-G_j^u$ in Lemma \ref{martingaleL2bounds} we have that
\begin{equation*}
\mathbb{E}_\rho^\beta[(\mathcal{M}_{t}^{n}(f))^2]\leq{t\Big\{2\chi(\rho)\Big[ \,\frac{1}{n}\sum_{x\neq
-1}\!\!
\big(\nabla_nf(\pfrac{x}{n})\big)^{2} +  n^{1-\beta}
\big(f(\pfrac{0}{n})-f(\pfrac{-1}{n})\big)^{2}\Big]+O_f({1}/{j})\Big\}}.
\end{equation*}
Hence, by the definition of $f$ we can bound the previous expression by $2\chi(\rho)/j$, which vanishes as
$j\rightarrow{+\infty}$.
On the other hand, taking
$f=H_u-G_j^u$ in Lemma \ref{integralL2bounds}, we get to
\begin{equation*}
\mathbb{E}_\rho^\beta[(\mc I_{t}^{n}(f))^2]\leq{80\,t\,\Big\{\chi(\rho)\Big[ \frac{1}{n}\sum_{x\neq -1}
\!\!\!\big(\nabla_nf(\pfrac{x}{n})\big)^{2} +  n^{1-\beta}
\big(f(\pfrac{0}{n})-f(\pfrac{-1}{n})\big)^{2}\Big]\!+\!O_f(1/j)\Big\}},
\end{equation*}
which can be bounded from above by $80\,t\,\chi(\rho)/j$ and vanishes as $j\rightarrow{+\infty}$,
finishing the proof of this lemma.
\end{proof}

\begin{proof}[Proof of Theorem \ref{current clt}]
The proof follows from the previous lemma and Theorem \ref{flu1}. We start with some considerations that work for all
$\beta\in[0,{+\infty}]$.

First we observe that the functions $G^u_j$ do not belong to $\mc S_\beta(\bb R)$. So, we fix $j\in{\mathbb{N}}$ and approximate each  $G^u_j$, in the $L^2(\mathbb R)$-norm with respect to the Lebesgue measure, by a sequence of smooth functions of compact support,
let us say  $H^u_{k,j}$. Moreover, we choose $H^u_{k,j}$ constant in a neighborhood of zero, which
ensures that $H^u_{k,j}\in\mc S_\beta(\bb R)$. For these functions we have convergence of the density fields. Moreover, for  fixed $t\geq{0}$ we have that
 \begin{equation*}
 \begin{split}
 \mathbb{E}_\rho^\beta\big[\big(\mathcal{Y}_t^n(H^u_{k,j})-\mathcal{Y}_t^n(G^u_j)\big)^2\big]&=\mathbb{E}_\rho^\beta\big[\big(\mathcal{Y}_t^n(H^u_{k,j}-G^u_j)\big)^2\big]\\
 &=\mathbb{E}_\rho^\beta\Big[\Big(\frac{1}{\sqrt {n}}\sum_{x\in{\bb Z}}(H^u_{k,j}-G^u_j)\big(\pfrac{x}{n}\big)\bar{\eta}_{tn^2}(x)\Big)^2\Big]\\
 &\leq{\chi(\rho)\|H^u_{k,j}
-G_j^u\|_2^2}\,,
\end{split}
 \end{equation*}
 which vanishes as $k\to +\infty$, by hypothesis.
 Hence $\mathcal{Y}_t^n(H^u_{k,j})$ converges to $\mathcal{Y}_t^n(G^u_j)$ in $L^2(\mathbb{P}_\rho^\beta)$, as $k\to{{+\infty}}$. By Theorem \ref{flu1}, we have that $\mathcal{Y}_t^n(H^u_{k,j})$ converges  to $\mathcal{Y}_t(H^u_{k,j})$ in distribution, as
$n\to{+\infty}$.
On the other hand, since for all $H,G\in\mathcal{S}_\beta(\mathbb{R})$,
\begin{equation}\label{T_beta}
 \mathbb{E}_\rho^\beta[\mathcal{Y}_t(H)\mathcal{Y}_s(G)]=\chi(\rho)\int_{\mathbb R}T^\beta_{t-s}H(v)G(v)dv\,,
\end{equation}
and since
 $\mathcal{Y}_t$ is linear, we have
 \begin{equation*}
 \begin{split}
 \mathbb{E}_\rho^\beta[(\mathcal{Y}_t(H^u_{k,j})-\mathcal{Y}_t(G^u_j))^2]&=\mathbb{E}_\rho^\beta[(\mathcal{Y}_t(H^u_{k,j}-G^u_j))^2]\\
&= \chi(\rho)\|H^u_{k,j}-G_j^u\|_2^2\,.
 \end{split}
\end{equation*}

Therefore $\mathcal{Y}_t(H^u_{k,j})$ converges to $\mathcal{Y}_t(G^u_j)$ in $L^2$, as $k\to{+\infty}$.
As a consequence of the previous results, $\mathcal{Y}_t^n(G^u_j)$ converges to $\mathcal{Y}_t(G^u_j)$ in distribution, as $n\to{+\infty}$.
By the previous lemma, $\{\mathcal{Y}_{t}^{n}(G^u_{j})
-\mathcal{Y}_{0}^{n}(G^u_{j})\}_{j\in\bb N}$ is a Cauchy sequence uniformly in $n$. Then, $\{\mathcal{Y}_{t}(G^u_{j})
-\mathcal{Y}_{0}(G^u_{j})\}_{j\in{\mathbb{N}}}$
is also a Cauchy sequence and converges, as $j\to +\infty$, to some random variable with gaussian distribution.
We denote such limit by $\mathcal{Y}_{t}(H_u)-\mathcal{Y}_{0}(H_u)$. Therefore, the normalized current ${J}_{u}^{n}(t)/\sqrt n$
converges to a gaussian random variable, which
formally reads as $\mathcal{Y}_{t}(H_{u})-\mathcal{Y}_{0}(H_{u}),$
where $\mathcal{Y}_t$ is the solution of the Ornstein-Uhlenbeck equation \eqref{OU}. Since the distributions of
$\mathcal{Y}_{t}(H_{u})$ are gaussian, this implies the limit current to be gaussian distributed.

The same argument can be applied to show the
same result for any vector
$({J}^n_u(t_{1}),..,{J}^n_u({t_{k}}))$.

We claim that to compute the covariance,  it is enough to compute the variance. Reversibility plus a simple computation
together with \eqref{T_beta} yields
\begin{equation}\label{var OU}
\begin{split}
\mathbb{E}_\rho^\beta[({J}_u({t}))^2]=&2\mathbb{E}_\rho^\beta[\mathcal{Y}_0(H_u)(\mathcal{Y}_{0}(H_{u})-\mathcal{Y}_{t}(H_{u}))]\\
=& 2\chi(\rho)\langle H_u,H_u-T^\beta_tH_u\rangle\,,
\end{split}
\end{equation}
where  $\<\cdot,\cdot\>$ denotes the inner product in $L^2(\bb R)$.
Above we used \eqref{T_beta} despite $H_u$ is not in $\mathcal{S}_\beta(\mathbb{R})$. Nevertheless, by approximating arguments
as above one can get that equality for $H_u$. Then, linearity shows that the covariance can be written as
\begin{equation*}
\begin{split}
\mathbb{E}_\rho^\beta[{J}_{u}(t){J}_u(s)]&=\chi(\rho)\Big[\langle H_u,H_u-T^\beta_tH_u\rangle +\langle
H_u,H_u-T^\beta_sH_u\rangle-\langle H_u,H_u-T^\beta_{t-s}H_u\rangle\Big]\\
&=\frac{1}{2}\Big\{\mathbb{E}_\rho^\beta[({J}_{u}(t))^2]+\mathbb{E}_\rho^\beta[({J}_{u}(s))^2]-\mathbb{E}_\rho^\beta[({J}_{u}(t-s))^2]\Big\}\,.
\end{split}
\end{equation*}
Therefore, we only need to compute the variance for each one of the regimes of $\beta$.

\quad

 $\bullet$ Case $\beta\in{[0,1)}$.

\smallskip

 Recalling  \eqref{sem heat eq}, we have that
\begin{equation*}
\begin{split}
\langle H_u,H_u-T^\beta_tH_u\rangle=&\int_{u}^{{+\infty}}\Big(1-\int_{u}^{+\infty}\frac{1}{\sqrt{4\pi
t}}e^{-\frac{(x-y)^2}{4t}}dy\Big)dx=\sqrt\frac{t}{\pi}\,.
\end{split}
\end{equation*}
From \eqref{var OU} we get
\begin{equation*}
\mathbb{E}_\rho^\beta[{J}_u({t}){J}_u(s)]=\chi(\rho)\Big(\sqrt{\frac{t}{\pi}}+\sqrt{\frac{s}{\pi}}-\sqrt{
\frac{t-s}{\pi} } \Big).
\end{equation*}

\quad

$\bullet$ Case $\beta=1$.

\smallskip

Recalling Proposition \ref{prop23},  we have that $\langle H_u,H_u-T^\beta_tH_u\rangle$ is equal to
\begin{equation*}
\begin{split}
&\int_{u}^{{+\infty}}\Big(1-\int_{-\infty}^{-u}\frac{1}{2\sqrt{4\pi t}}e^{-\frac{(x-y)^2}{{4t}}}dy-\int_{u}^{+\infty}\frac{1}{2\sqrt{4\pi
t}}e^{-\frac{(x-y)^2}{{4t}}}dy\\
-&e^{2\alpha x}\int_{x}^{+\infty}\frac{e^{-2\alpha z}}{2}\int_{u}^{+\infty}\Big\{\frac{z-y+4\alpha t}{2t\sqrt{4\pi
t}}e^{-\frac{(z-y)^2}{{4t}}}+\frac{z+y-4\alpha t}{2t\sqrt{4\pi t}}e^{-\frac{(z+y)^2}{4t}}\Big\}dy\,dz\Big)dx\,,\\
\end{split}
\end{equation*}
which can be rewritten as
\begin{equation*}
\begin{split}
&\int_{u}^{{+\infty}}\Big(\frac{1}{2}+\int_{-u}^{-u}\frac{1}{2\sqrt{4\pi t}}e^{-\frac{(x-y)^2}{4t}}dy\\
&-e^{2\alpha x}\int_{x}^{+\infty}\!\!\!\frac{e^{-2\alpha z}}{2}\Big\{-\int_{z-u}^{z+u}\!\!\frac{v}{2t\sqrt{4\pi
t}}e^{-\frac{v^2}{{4t}}}dv+2\alpha-2\alpha \Phi_{2t}(z-u)-2\alpha\Phi_{2t}(z+u)\Big\}dz\Big)dx.
\end{split}
\end{equation*}
A long but elementary computation shows that
\begin{equation*}
\langle H_u,H_u-T^\beta_tH_u\rangle=\sqrt\frac{t}{\pi}+\frac{\Phi_{2t}(2u+4\alpha t)e^{4\alpha u} e^{4\alpha^2 t}}{2\alpha}-\frac{\Phi_{2t}(2u)}{2\alpha}\,,
\end{equation*}
which from \eqref{var OU} is enough to conclude.

\quad

 $\bullet$ Case $\beta\in{(1,{+\infty}]}$.

\smallskip

Recalling \eqref{sem heat eq neu}, we have that
\begin{equation*}
\begin{split}
\langle H_u,H_u-T^\beta_tH_u\rangle =&\int_{u}^{+\infty}\Big(1-\int_{u}^{+\infty}\frac{1}{\sqrt{4\pi
t}}e^{-\frac{(x-y)^2}{4t}}dy-\int_{u}^{+\infty}\frac{1}{\sqrt{4\pi t}}e^{-\frac{(x+y)^2}{4t}}dy\Big)dx\\
=&\sqrt\frac{t}{\pi}\Big[1-e^{-u^2/t}\Big]+2u\Phi_{2t}(2u)\,,
\end{split}
\end{equation*}
which from \eqref{var OU} concludes the proof.
\end{proof}

\begin{proof}[Proof of Corollary \ref{limit robin current}]
 In order to prove the result notice that gaussian processes are
characterized by its covariance,  and the limit of the covariance guarantees the convergence of the processes in the sense of
finite dimensional distributions. Thus, it is sufficient to show that
\begin{equation*}
\lim_{\alpha\rightarrow{0}}\frac{\Phi_{2t}(2u+4\alpha
t)e^{4\alpha u+4\alpha^2t}}{2\alpha}-\frac{\Phi_{2t}(2u)}{2\alpha}=2u\Phi_{2t}(2u)-\sqrt{\frac{t}{\pi}}e^{-u^2/t}
\end{equation*}
and
\begin{equation*}
\lim_{\alpha\rightarrow{+\infty}}\frac{\Phi_{2t}(2u+4\alpha t)e^{4\alpha u+4\alpha^2t}}{2\alpha}-\frac{\Phi_{2t}(2u)}{2\alpha}=0.
\end{equation*}
The first limit comes out by L'Hôpital's Rule and  the second one is consequence of
the estimate $\int_a^{+\infty}e^{-x^2/2}dx\leq{\frac{1}{a}e^{-a^2/2}}$, for $a\in{\mathbb{R}}$.
\end{proof}

\section{Some useful $L^2$ estimates}\label{s7}

In this section we prove what we call {\em Local Replacement} which is fundamental in characterizing the limit points of the
density fluctuation field.

For a function $g\in L^2(\nu_\rho)$, we denote by $\mc D_n(g)$  the Dirichlet form of the function $g$,  defined as
$\mc D_n (g) \;=\; -\int g(\eta) \mc L_n g(\eta) \,\nu_\rho(d\eta) .$
 An elementary computation shows
that
\begin{equation}\label{dirichlet}
\mc D_n (g) \;=\; \sum_{x\in \bb Z} \frac{ \xi_{x,x+1}^n}{2}
\int  \Big( g(\eta^{x,x+1}) -
g(\eta) \Big)^2 \,\nu_\rho(d\eta)\;.
\end{equation}

Recall from \eqref{empirical average} that
\begin{equation*}
 \bar \eta^{\ell}(x)\;=\;\frac{1}{\ell}\sum_{y=x}^{x+\ell-1}\bar \eta(y).
\end{equation*}
 \begin{lemma}[{\em{Local Replacement}}] \label{2orderBG}
\quad

 For $\beta\in[0,{+\infty}]$, for $\ell\geq{1}$ and for $x=-1$ it holds that
\begin{equation*}
\mathbb{E}_\rho^\beta\Big[\Big(\int_{0}^t \{\bar \eta_{sn^2}(x)-\bar
\eta_{sn^2}^\ell(x)\}ds\Big)^2\Big]\leq{\frac{80t}{n^2}\chi(\rho)\Big(\alpha n^{\beta}+\ell\Big)}.
\end{equation*}
\end{lemma}

In order to prove last lemma, we use the following result.

 \begin{lemma}\label{usefullema}
For $\beta\in[0,{+\infty}]$, for $g\in{L^2(\nu_\rho)}$, for a constant $A>0$ and for $x=-1$, it holds that
\begin{equation*}
\int \{\bar \eta(x)-\bar \eta^\ell(x)\}g(\eta)\nu_\rho(d\eta)\leq{A\chi(\rho)(\alpha n^\beta+\ell)+ \frac{1}{A}\mathcal{D}_n(g)},
\end{equation*}
where
$\mathcal{D}_n(g)$ is the Dirichlet form, see \eqref{dirichlet}.
\end{lemma}

\begin{proof}
In order to prove the previous lemma, we notice that by the definition of the empirical average given in \eqref{empirical
average}, we are able to write the integral in the statement of the lemma as
\begin{equation*}
\frac{1}{\ell} \sum_{y=x}^{x+\ell-1}\sum_{z=x}^{y-1} \int\{\eta(z)-\eta(z+1)\}g(\eta)\nu_\rho(d\eta).
\end{equation*}
Writing the previous expression as twice its half and performing the change of variables $\eta\mapsto\eta^{z,z+1}$, for which the
measure $\nu_\rho$ is invariant, we get to
\begin{equation*}
\frac{1}{2\ell}\sum_{y=x}^{x+\ell-1}\sum_{z=x}^{y-1}\int (\eta(z)-\eta(z+1))(g(\eta)-g(\eta^{z,z+1}))\nu_\rho(d\eta).
\end{equation*}
Now, by the Cauchy-Schwarz inequality we bound last expression by
\begin{equation*}
\begin{split}
&\frac{1}{2\ell} \sum_{y=x}^{x+\ell-1}\sum_{z=x}^{y-1}\frac{A}{\xi^n_{z,z+1}} \int(\eta(z)-\eta(z+1))^2\nu_\rho(d\eta)\\+&
\frac{1}{2\ell} \sum_{y=x}^{x+\ell-1}\sum_{z=x}^{y-1}\frac{\xi_{z,z+1}^n}{A} \int(g(\eta)-g(\eta^{z,z+1}))^2\nu_\rho(d\eta).
\end{split}
\end{equation*}
To finish the proof it is enough to recall  \eqref{dirichlet}.

\end{proof}

\begin{proof}[Proof of Lemma \ref{2orderBG}.]

By  Proposition A1.6.1 of \cite{kl} we have that
\begin{equation*}
\mathbb{E}_\rho^\beta\Big[\Big(\int_{0}^t \{\bar \eta_{sn^2}(x)-\bar \eta_{sn^2}^\ell(x)\}ds\Big)^2\Big]\leq{20\,t\|\bar
\eta(x)-\bar \eta^\ell(x)\|_{-1}^2}
\end{equation*}
\begin{equation*}
\begin{split}
&=20 \,t\sup_{g\in{L^2(\nu_\rho)}}\Big\{2\int \{\bar \eta(x)-\bar \eta^\ell(x)\}g(\eta)\nu_\rho(d\eta) - n^2\mathcal{D}_n(g)\Big\}\\
&\leq 20 \,t\sup_{g\in{L^2(\nu_\rho)}} \Big\{ 2A\chi(\rho)(\alpha n^\beta+\ell)+ \frac{2}{A}\mathcal{D}_n(g) -n^2\mathcal{D}_n(g)\Big\}.
 \end{split}
\end{equation*}
In last inequality we used the Schwarz inequality together with the previous lemma.
Taking $2/A=n^2$ the claim follows.
\end{proof}

\begin{remark}\label{useful remark}
 Using the same arguments as above, we obtain the statement of Lemma \ref{2orderBG} and Lemma \ref{usefullema} for $x=0$ exactly with the same bounds as for $x=-1$ but removing the term  $\alpha n^\beta$. This is a consequence of the fact that in this case we do not cross the slow bond.
\end{remark}

\begin{lemma}\label{martingaleL2bounds}
Fix $H\in\mc S_\beta(\bb R)$. For $\beta\in[0,{+\infty}]$ and for any $t\geq{0}$
 \begin{equation}\label{var_quad}
\begin{split}
 & \mathbb{E}_\rho^\beta\big[(\mc M^n_t(H))^2\big]= t\Big\{2\chi(\rho)\Big[ \pfrac{1}{n}\sum_{x\neq -1}
\big(\nabla_nH(\pfrac{x}{n})\big)^{2} + \alpha n^{1-\beta}
\big(H(\pfrac{0}{n})-H(\pfrac{-1}{n})\big)^{2}\Big]+O_H(\pfrac{1}{n})\Big\}
\end{split}
 \end{equation}
 and
\begin{equation*}
\lim_{n\to {+\infty}}\mathbb{E}_\rho^\beta[(\mc M_t^n(H))^2]=
2\chi(\rho)\,t \|\B H \|_{2,\beta}^2,
\end{equation*}
where $\mc M^n_t(H)$ is the martingale defined in \eqref{martingale decomposition}.
\end{lemma}
\begin{proof}
The quadratic variation of $\mc M^n_t(H)$  is given by
\begin{equation*}
\<\mc M^n(H)\>_t=\int_{0}^{t} n^{2} \Big[ \mc L_{n} \mc Y^n_s(H)^2-2 \mc Y^n_s(H)\mc L_{n} \mc Y^n_s(H)\Big] ds.
\end{equation*}
A simple computation shows that
\begin{equation}\label{var_quad1}
\begin{split}
\<\mc M^n(H)\>_t =&\int_{0}^{t} \frac{1}{n}\sum_{x\neq -1}(\eta_{sn^2}(x)-\eta_{sn^2}(x+1))^2
\Big[n\big(H(\pfrac{x+1}{n})-H(\pfrac{x}{n})\big)\Big]^{2} ds \\
+&\int_{0}^{t} \alpha n^{1-\beta}(\eta_{sn^2}(-1)-\eta_{sn^2}(0))^2
\big(H(\pfrac{0}{n})-H(\pfrac{-1}{n})\big)^{2} ds .
\end{split}
\end{equation}
To finish the first claim of the lemma is enough to take expectation with respect to $\nu_\rho$ in last expression.

Now, we prove the second claim.
Since for all $\beta\in[0,{+\infty}]$, $H\in\mc S(\bb R\backslash\{0\})$,  the first term on the right side of \eqref{var_quad}
converges to
 $2\chi(\rho)\,t\Vert \B H \Vert^2_2$, as  $n\to{+\infty}$. To finish the proof, it is enough to analyze the second term on the right
side of \eqref{var_quad}.
For $\beta\in[0,1)$ since $H\in\mc S(\bb R)$, then by Taylor's expansion is it easy to check that the second term above is of order
$O_H(n^{-\beta})$, which also vanishes as $n\rightarrow{+\infty}$.
For $\beta\in(1,+\infty]$, the second term on the right side of \eqref{var_quad} is bounded from above by $tn^{1-\beta}4\Vert
H\Vert_\infty^2$, which vanishes as $n\rightarrow{+\infty}$. Finally, for $\beta=1$, we use Taylor's expansion and the fact that
$\alpha\{H(0^+)-H(0^-)\}=H^{(1)}(0^-)=H^{(1)}(0^+)$ to show that it converges, as $n\to{+\infty}$, to
$2\chi(\rho)\,t \big(H^{(1)}(0^+)\big)^{2}$. This concludes the proof.
\end{proof}

\begin{corollary}
Fix  $H\in\mc S_\beta(\bb R)$. For $\beta\in[0,{+\infty}]$ and for any $t\geq{0}$

 \begin{equation}\label{var_quad_bound}
\begin{split}
 &\vert \<\mc M^n(H)\>_t\vert\leq t\Big\{\frac{1}{n}\sum_{x\neq -1}
\big(H^{(1)}(\pfrac{x}{n})\big)^{2} +  n^{1-\beta}
\big(H(\pfrac{0}{n})-H(\pfrac{-1}{n})\big)^{2}+O_H(\pfrac{1}{n})\Big\}  .
\end{split}
 \end{equation}

\end{corollary}
\begin{proof}
 It is enough to use the triangular inequality in equation \eqref{var_quad1}, together with the fact that
$(\eta_{sn^2}(x)-\eta_{sn^2}(x+1))^2\leq 1$, for  all $x\in\bb Z$ and $s\geq 0$.
 \end{proof}


\begin{lemma}\label{lema_novo}
Let $g\in L^2(\nu_\rho)$ and $\{F_n\}_{n\in{\bb N}}$ a sequence of functions $F_n:\bb R\to\bb R$.
For any constant $A>0$,
\begin{equation*}
 \int \sum_{x\in\bb Z}F_n(\pfrac{x}{n})\{\eta(x)-\eta(x+1)\}g(\eta)\nu_\rho(d\eta)\,
\leq\, A\chi(\rho)\sum_{x\in\bb Z}\frac{1}{\xi_{x,x+1}^n}\big(F_n(\pfrac{x}{n})\big)^2
+\frac{1}{A}\,\mc D_n(g),
\end{equation*}
where $\mc D_n(g)$ is the Dirichlet form given in \eqref{dirichlet}.
\end{lemma}
 \begin{proof}
Rewriting the expression above as twice the half and making the transformation $\eta\mapsto \eta^{z,z+1}$ (for
which the probability $\nu_\rho$ is invariant), we have
\begin{equation*}
\begin{split}
& \int \sum_{x\in\bb Z}F_n(\pfrac{x}{n})\{\eta(x)-\eta(x+1)\}g(\eta)\nu_\rho(d\eta)\\
=&
\frac{1}{2} \int \sum_{x\in\bb Z}F_n(\pfrac{x}{n})\{\eta(x)-\eta(x+1)\}\{g(\eta)-g(\eta^{x,x+1})\}\nu_\rho(d\eta)
.
\end{split}
\end{equation*}
By Cauchy-Schwarz's inequality, for any $A>0$, we bound the previous expression from above by
\begin{equation*}
\begin{split}
 &  \frac{1}{2}\sum_{x\in\bb Z}\frac{A}{\xi_{x,x+1}^n}\big(F_n(\pfrac{x}{n})\big)^2
\int\{\eta(x)-\eta(x+1)\}^2\,\nu_\rho(d\eta)\\\
 +\,&  \frac{1}{2}\sum_{x\in\bb Z}
\frac{\xi_{x,x+1}^n}{ A}\int \{g(\eta)-g(\eta^{x,x+1})\}^2\,\nu_\rho(d\eta).
\end{split}
 \end{equation*}
Recalling \eqref{dirichlet}, the  proof finishes.
 \end{proof}

\begin{lemma}\label{integralL2bounds}
Fix $H\in\mc S_\beta(\bb R)$. For $\beta\in[0,{+\infty}]$ and for any $t\geq{0}$
\begin{equation}\label{I^2}
\mathbb{E}_\rho^\beta \Big[\big(\mc I_{t}^{n}(H)\big)^2\Big]\leq 80 \,t\,\chi(\rho)\Big\{\frac{1}{n}\sum_{x\neq
-1}\big(\nabla_nH(\pfrac{x}{n})\big)^2
+n^{1-\beta}\big[H\big(\pfrac{0}{n}\big)-H\big(\pfrac{-1}{n}\big)\big]^2\Big\},
\end{equation}
where $\nabla_n H\big(\pfrac{x}{n}\big)=n\big[H\big(\pfrac{x+1}{n}\big)-H\big(\pfrac{x}{n}\big)\big]$ and
\begin{equation*}
\limsup_{n\to{+\infty}}\,\mathbb{E}_\rho^\beta \Big[\big(\mc I_{t}^{n}(H)\big)^2\Big]\leq
80\, t\,\chi(\rho)\|\B H\|_{2,\beta}^2,
\end{equation*}
where
$\mc I_{t}^{n}(H)$ was defined in \eqref{I}.
\end{lemma}

\begin{proof}
Recall the definition of $\mc I_{t}^{n}(H)$ given in \eqref{I}. A simple computation shows that
\begin{equation*}
\begin{split}
 \mc I_{t}^{n}(H)=&\int_0^t\sqrt{n}\sum_{x\neq -1,0}\Big\{\nabla_n H\big(\pfrac{x}{n}\big)-\nabla_n
H\big(\pfrac{x-1}{n}\big)
\Big\}\eta_{sn^2}(x)\,ds\\
+&\int_0^t\sqrt{n}\Big\{\nabla_n H\big(\pfrac{0}{n}\big)\eta_{sn^2}(0)-\nabla_n
H\big(\pfrac{-2}{n}\big)\eta_{sn^2}(-1)
\Big\}ds\\
+&
\int_0^tn^{3/2-\beta}\big[H\big(\pfrac{0}{n}\big)-H\big(\pfrac{-1}{n}\big)\big]\{\eta_{sn^2}(-1)-\eta_{sn^2}(0)\}ds\
.
\end{split}
\end{equation*}

Last expression can be rewritten as
\begin{equation*}
\begin{split}
&\int_0^t\sum_{x\in{\bb Z}}F_n\big(\pfrac{x}{n}\big)\{\eta_{sn^2}(x)-\eta_{sn^2}(x+1)\},
\end{split}
\end{equation*}
where
\begin{equation*}
 F_n\big(\pfrac{x}{n}\big)= \left\{
\begin{array}{ll}
n^{3/2-\beta}\big[H\big(\pfrac{0}{n}\big)-H\big(\pfrac{-1}{n}\big)\big],& \mbox{if}\,\,x= -1,\\ \\
\sqrt{n}\,\,\nabla_n H\big(\pfrac{x}{n}\big),&\,\,\mbox{otherwise}.\\
\end{array}
\right.
\end{equation*}
By Proposition A1.6.1 of \cite{kl}, we have that
\begin{equation*}
\begin{split}
&\mathbb{E}_\rho^\beta \Big[\big(\mc I_{t}^{n}(H)\big)^2\Big]\leq 20 t\,\sup_{g\in L^2(\nu_\rho)}\Bigg\{2\int
\sum_{x\in{\bb Z}}F_n\big(\pfrac{x}{n}\big)\{\eta(x)-\eta(x+1)\}g(\eta)\nu_\rho(d\eta)
-n^2\mc D_n(g)\Bigg\}.
\end{split}
\end{equation*}
By Lemma \ref{lema_novo}, last expression is bounded from above by
\begin{equation*}
\begin{split}
20 t\,\sup_{g\in L^2(\nu_\rho)}\Bigg\{2A\chi(\rho)\sum_{x\in\bb Z}\frac{1}{\xi_{x,x+1}^n}\big(F_n(\pfrac{x}{n})\big)^2
+\frac{2}{A}\,\mc D_n(g)
-n^2\mc D_n(g)\Bigg\}.
\end{split}
\end{equation*}
Taking $A=\pfrac{2}{n^2}$ and by the definition of $F_n$ the proof of the first claim ends.

To prove the second one, we notice the following. The first term on the right hand side of \eqref{I^2} converges to
$80t\,\chi(\rho)\Vert \B H \Vert_2^2$.
The second term on the right hand side of \eqref{I^2} can be analyzed as in the proof of Lemma \ref{martingaleL2bounds}.
\end{proof}

\section*{Acknowledgements}

The authors thank the warm hospitality of IMPA (Brazil)  where this work was
initiated  and to CMAT (Portugal) where this work was finished. The authors thank FCT (Portugal) and CAPES (Brazil) for the financial support through the
research project "Non-Equilibrium Statistical Mechanics of Stochastic
Lattice Systems".

P.G. thanks FCT (Portugal) for support through the research
project ``Non-Equilibrium Statistical Physics" PTDC/MAT/109844/2009.
PG thanks the Research Centre of Mathematics of the University of
Minho, for the financial support provided by ``FEDER" through the
``Programa Operacional Factores de Competitividade  COMPETE" and by
FCT through the research project PEst-C/MAT/UI0013/2011.

T.F. was supported through a grant "BOLSISTA DA CAPES - Brasília/Brasil" provided by CAPES (Brazil).


\begin{thebibliography}{10}


\bibitem{dg} P. Dittrich  and J. G\"artner. {\em A central limit theorem
  for the weakly asymmetric simple exclusion process}. Math. Nachr., 15, 75--93, (1991).

\bibitem{fjl} A. Faggionato, M. Jara, and C. Landim. {\em Hydrodynamic behavior of one dimensional subdiffusive exclusion processes with random conductances}. Probab. Th. and Rel. Fields, 144, no. 3-4, 633--667, (2009).

\bibitem{fsv} J. Farfan, A.B. Simas,  and F. J. Valentim. {\em Equilibrium fluctuations for exclusion processes with conductances in random environments}. Stochastic Process. Appl., 120, no. 8, 1535--1562, (2010).

\bibitem{fgn}
T. Franco, P. Gon\c calves, and A. Neumann.
{\em Hydrodynamical behavior of symmetric exclusion with slow bonds}. Ann. Inst. H. Poincaré Probab. Statist. Volume 49, Number 2, 402--427 (2013).

 \bibitem{fgn2} T. Franco, P. Gon\c calves, and A. Neumann.
 {\em Phase Transition of a Heat Equation with Robin's Boundary Conditions and Exclusion Process}. arXiv:1210.3662.

\bibitem{fl} T. Franco and C. Landim. {\em Hydrodynamic Limit of Gradient Exclusion Processes with conductances}. Arch. Ration. Mech. Anal., 195, no. 2, 409--439, (2010).

\bibitem {g}
P. Gon\c calves. {\em Central Limit Theorem for a Tagged Particle in Asymmetric Simple Exclusion}. Stochastic Process and their Applications, 118, 474--502, (2008).

\bibitem{gj}
P. Gon\c calves and M. Jara. {\em Universality of KPZ equation}, arXiv:1003.4478, (2010).


\bibitem{HS}
R. Holley and D. Stroock. {\em Generalized {O}rnstein-{U}hlenbeck processes and infinite particle
  branching {B}rownian motions}. Publ. Res. Inst. Math. Sci.,  14, no. 3, 741--788, (1978).


\bibitem{jl}
M. Jara and C. Landim. {\em Non Equilibrium Central Limit Theorem for a Tagged Particle in Symmetric Simple Exclusion}. Annals Inst. H. Poincar\'e (B) Probab. and Statist., 42, 567--577, (2006).


\bibitem{kl}

C. Kipnis and C. Landim, {\em Scaling limits of interacting
  particle systems}. Grundlehren der Mathematischen Wissenschaften
  [Fundamental Principles of Mathematical Sciences], 320.
  Springer-Verlag, Berlin. (1999).

\bibitem{Mit}
I. Mitoma, {\em{ Tightness of probabilities on {$C([0,1];{\mathcal
  S}\sp{\prime} )$} and {$D([0,1];{\mathcal S}\sp{\prime} )$}}}. Ann. Prob.,  11, no. 4, 989--999, (1983).


\bibitem{ps}
M. Peligrad and S. Sethuraman, {\em  On fractional Brownian motion limits in one dimensional nearest-neighbor symmetric simple exclusion}. ALEA Lat. Am. J. Probab. Math. Stat., 4, 245--255, (2008).

  \bibitem{rv} M. E. Vares and H. Rost {\em Hydrodynamics of a One-Dimensional Nearest Neighbor Model}. AMS Contemporary
Mathematics, 41, 329--342, (1985).


\bibitem{rs}
M. Reed and B. Simon. {\em Functional Analysis,
Volume 1 (Methods of Modern Mathematical Physics)}. Academic Press, 1
edition, (1981).


\end{thebibliography}
\end{document}